\documentclass[onecolumn]{autart}
\usepackage{}    

\usepackage{tikz} 
\usetikzlibrary{shapes,arrows,chains,backgrounds,fit,decorations.pathreplacing}
\usepackage[europeanresistors,americaninductors]{circuitikz}
\usetikzlibrary{shapes,arrows}
\usepackage{amsmath,bm,times}

\usepackage{mathrsfs}
\usepackage{bbm}    

\usepackage{lineno,hyperref}

\usepackage{graphicx}           
\usepackage{dcolumn}            
\usepackage{bm}                 

\usepackage{graphics}           
\usepackage{epsfig}             
\usepackage{times}              
\usepackage{amsmath}
\usepackage{amssymb}
\usepackage{amsfonts}
\usepackage{fancyhdr}
\usepackage{color}
\usepackage{subfigure}
\usepackage{enumerate}
\usepackage{cite}
\usepackage{wrapfig}
\usepackage{url}
\allowdisplaybreaks[4] 

\newtheorem{theorem}{Theorem} %
\newtheorem{lemma}{Lemma}

\newtheorem{proposition}{Proposition}

\newtheorem{definition}{Definition}

\newenvironment{proof}{{\it Proof: \enspace}}{\hfill $\blacksquare$\par}

\begin{document}

%
\begin{frontmatter}

\title{Hausdorff measure of the free boundary for the $p$-obstacle problem with subcritical exponents} 


\author{Jing Yu}\ead{jingyu@my.swjtu.edu.cn}\ and 
\author{Jun Zheng}\ead{zhengjun2014@aliyun.com}

\address{{School of Mathematics}, Southwest Jiaotong University, Chengdu, Sichuan 611756, P.~R.~of~China}  

\begin{abstract}                           
This paper investigates a class of $p$-obstacle  problems with subcritical exponents having the form
\begin{align}
	\mathrm{div}\left( a(x)|\nabla u|^{p-2}\nabla u\right) =m_1\chi_{\{u>0\}}-m_2u^{\lambda-1}\chi_{\{u>0\}} \ \text{in}\ \Omega,\notag
\end{align}
where $\Omega  $ is a smooth bounded domain in $ \mathbb{R}^N (N \geq 2)$,  $m_1,m_2$ are positive constants,   the coefficient function $a \in C^2(\Omega)$ has a positive lower bound, and $
2 \leq p < \lambda <p^*:= \frac{Np}{N-p}$ when    $p<N$ and $N \geq 3$, or
$2\leq p < \lambda <+\infty$ when $ N = 2$.
By using    the mountain-pass lemma, combined with the penalty method, we first establish the existence of non-negative weak solutions.  Then, using the De Giorgi-Nash iteration, we prove the $L^\infty$ bound and   local $C^{1,\alpha}$ continuity for the solutions. In addition, we prove local porosity of the free boundary based on the optimal growth and non-degeneracy of solutions near the free boundary.  Furthermore, by means of Lebesgue measure estimates for gradient level sets, we show that    at least one solution  corresponds the free boundary having locally finite $(N-1)$-dimensional Hausdorff measure.
\end{abstract}

\vspace*{-12pt}
\end{frontmatter}
\endNoHyper


\section{Introduction}\label{Sec.1}
Free boundary problems described by partial differential equations (PDEs) of elliptic type arise naturally in a wide range of physical models and variational settings, and constitute a central subject in the theory of PDEs. Typical examples include the obstacle problem, Bernoulli-type free boundary problems, chemical reaction  equations,  among others. A crucial theme of studying the free boundary problems is proving the existence and regularity of solutions and analyzing the geometric properties of the free boundary such as its  porosity and Hausdorff measure.

Recently, the authors of \cite{Wang:2025} studied a superlinear obstacle-type free boundary problem of the form
\begin{align}
	\begin{cases}
		\Delta u = \left(1 - u^{\lambda-1}\right) \chi_{\{u > 0\}}  \text{ in } \Omega, \\
		u \geq 0  \text{ in } \Omega, \\
		u \in W_0^{1,2}(\Omega),
	\end{cases}\notag
\end{align}
where $\Omega  $ is a smooth bounded domain in $ \mathbb{R}^N (N \geq 2)$, $\chi_{\{u>0\}}(\cdot)$ is  the standard characteristic function,
and the exponent $\lambda$ satisfies
\begin{align}
	2 < \lambda <
	\begin{cases}
		2^*:= \frac{2N}{N-2}, & N \geq 3, \\
		+\infty, & N = 2.
	\end{cases}\notag
\end{align}
The main challenge in studying regularity theory for such a problem
comes from the coupling of the superlinear term $u^{\lambda-1}  $ and the characteristic function $ \chi_{\{u > 0\}} $. This coupling disrupts the coercivity of the associated energy functional (thus ruling out standard minimization arguments for existence). As a consequence, classical variational methods cannot be applied straightforwardly. To overcome this obstacle, the authors of  \cite{Wang:2025} employed a penalty method, combined with the mountain pass lemma, to establish the existence of nontrivial solutions, and further demonstrated that these solutions possess locally optimal  regularity by means of a suitable monotonicity formula combined with a bootstrap iteration argument. Moreover, non-degeneracy of solutions near the free boundary and local porosity of the free boundary were established. Furthermore, it was shown that the free boundary possesses locally finite $(N-1)$-dimensional Hausdorff measure  and a decomposition property of the free boundary was provided; see \cite{Wang:2025}.

Motivated by the superlinear obstacle-type free boundary problem studied in \cite{Wang:2025}, the aim of the present paper is to investigate a class of obstacle problems with  more involved operators and nonlinear structures. In particular, in this paper, quasilinear operators of the $p$-Laplace type with variable coefficients  are considered and the nonlinear term on the right-hand side of the equation allows for subcritical exponents, both of which bring  more complexities  in the establishment of solutions and analysis of free boundaries. More precisely, we  study the following problem:
\begin{align} \label{main}
	\begin{cases}
		\mathrm{div}\left( a(x)|\nabla u|^{p-2}\nabla u\right)  = m_1 \chi_{\{u>0\}} - m_2 u^{\lambda - 1} \chi_{\{u > 0\}}  \text{ in } \Omega, \\
		u \geq 0  \text{ in } \Omega, \\
		u \in W_0^{1,p}(\Omega),
	\end{cases}
\end{align}
where $\Omega  $ is a smooth bounded domain in $ \mathbb{R}^N (N \geq  2)$, the coefficient function $a \in C^2(\Omega)$ satisfies
\begin{align*}
	0<a_0\leq a \ \text{in}\ \Omega  \ \ \text{and}\ \ \|a\|_{C^2(\Omega)} \leq a_1
\end{align*}
with positive constants $a_0$ and $a_1$,   $m_1, m_2  $ are positive constants, and the exponent $\lambda$ satisfies
\begin{align}\label{p-condition}
	2\leq p < \lambda <
	\begin{cases}
		p^*:= \frac{Np}{N-p}, & N \geq 3, p<N, \\
		+\infty, & N = 2.
	\end{cases}
\end{align}
Formally,  equation \eqref{main} corresponds to the following energy functional
\begin{align}
	\mathcal{J}(u) = \int_{\Omega} \frac{a(x)}{p}|\nabla u|^p + m_1 u^+ - \frac{m_2}{\lambda}(u^+)^{\lambda}   \mathrm{d}x,\notag
\end{align}
where $u^+ = \max\{u, 0\}.$

It is worth mentioning that the authors of \cite{Karp:2000, Lee:2003, Zhao:2012,Zheng:2013} and \cite{Challal:2009, Challal:2012} investigated the porosity and Hausdorff measure of the free boundary for the $p$-Laplace type obstacle problem with $m_2 = 0$ in Sobolev spaces, and for the $A$-Laplace type obstacle problem with $m_2 = 0$ in Orlicz spaces, respectively. The porosity and Hausdorff measure of the free boundary for a $p$-Laplace obstacle problem associated with the Ginzburg-Landau equation, for which the coefficient of the  superlinear term is non-negative,  were studied in \cite{Hu:2025}.   It is also worth mentioning that  the porosity of the free boundary for the $p$-Laplace  non-zero obstacle problems  was established in  \cite{Andersson:2015}. Nevertheless, the regularity, in particular, the Hausdorff measure of the free boundary,  for the $p$-Laplace obstacle problem  with variable coefficients and  general subcritical terms remain unexplored.



In this paper, we first prove the existence and regularity of weak solutions of   the $p$-Laplace type obstacle problem   \eqref{main}, and then  study the local porosity and finite  $(N-1)$-dimensional Hausdorff measure of the free boundary $\Upsilon^+ := \partial\{u>0\} \cap \Omega$. 
%
The main contribution  of this paper is twofold:

\begin{itemize}
	\item By using a regularization and penalty method, we prove the existence and regularity of weak solutions for  a wider class of $p$-Laplace type obstacle problems, which extends relevant results obtained  in \cite{Wang:2025} for the obstacle problem governed by the Laplacian.
	
	\item We overcome the technical difficulties arising from  the general $p$, the variable coefficients, and   the nonlinearity of subcritical terms and show that for each solution the free boundary  is locally porous and  that   at least one solution corresponds the free boundary   having locally finite $(N-1)$-dimensional Hausdorff measure,   extending the  work in the literature \cite{Wang:2025,Karp:2000,Lee:2003,Zhao:2012,Hu:2025,Andersson:2015,Zheng:2013}.
\end{itemize}

The paper is organized as follows. In Section~\ref{Sec.2}, we present some   notations and technical lemmas used for the proofs of  main results. In Section~\ref{Sec.3}, we prove the existence  and uniform $L^\infty$ boundedness of solutions of equation \eqref{main}. In Section~\ref{Sec.4}, we establish the local $C^{0,\alpha}$ and $C^{1,\alpha}$ regularities of the solutions. In Section~\ref{Sec.5}, we prove local porosity of the free boundary based on the optimal growth and non-degeneracy of solutions near the free boundary. In Section~\ref{Sec.6}, we provide some auxiliary results for a corresponding penalized problem, following which we show that for at least one solution, the free boundary has locally finite $(N-1)$-dimensional Hausdorff measure in Section~\ref{Sec.7}.

\section{Notations and technical Lemmas}\label{Sec.2}
$\mathbb{R}$ denotes the set of real numbers, $\mathbb{N} := \{0,1,2,\dots\}$, and $\mathbb{N}^+ := \{1,2,3,\dots\}$.
$B_r(x)$ denotes the ball in    $ \mathbb{R}^N  $ with center $x$ and radius $r>0$. When the center is not emphasized, we simply write $B_r$   for a ball with radius $r$.
For a function $u$ defined on $\Omega$, let 
$\{u > 0\} := \{x \in \Omega : u(x) > 0\}$  and $\Upsilon^+ := \left(\partial\{x \in \Omega : u(x) > 0\}\right) \cap \Omega$.
For $m\in\mathbb{N}^+$ and   measurable set $E \subset \mathbb{R}^m$, $|E| := \mathcal{L}^m(E)$ denotes its $m$-dimensional Lebesgue measure, and $\mathcal{H}^m(E)$ its $m$-dimensional Hausdorff measure, respectively.

The following inequality ensures the monotonicity of the $p$-Laplacian.
\begin{lemma}[{\cite[p.13]{DiBenedetto:1993}}]\label{Lemma6}  For $ p\geq2$, it holds that
	\begin{align}
		\left( |\xi|^{p-2}\xi-|\eta|^{p-2}\eta \right)  \cdot (\xi-\eta) \geq  \frac{1}{2^p}|\xi-\eta|^{p}, \forall \xi,\eta \in \mathbb{R}^N.\notag
	\end{align}	
\end{lemma}

The following  result, which is  called the Du Bois-Reymond lemma, is a fundamental lemma    in calculus of variations.
\begin{lemma}[{\cite[Corollary 4.24]{Brezis:2010}}] \label{lem5}
	Let $f \in L^{1}_{\mathrm{loc}}(\Omega)$ satisfying $\int_{\Omega} f g \mathrm{d}x = 0$ for all $g \in C_{0}^{\infty}(\Omega)$. Then $f = 0$ a.e. in $\Omega$.
\end{lemma}

The following lemma will be used in the proof of the $L^\infty$ estimate of weak solutions of  the  obstacle problem \eqref{main}.
\begin{lemma}[\cite{Ladyzhenskaya:1968}] \label{lem1}
	Let $\{g_n\}_{n \in \mathbb{N}}$ be a sequence of non-negative numbers satisfying
	\begin{align}
		g_{n+1} \leq C D^n g_n^{1+\zeta}, \quad n \in \mathbb{N},\notag
	\end{align}
	where $C,\zeta > 0$ and $D > 1$ are constants not depending on $n$. If
	\begin{align}
		g_0 \leq C^{-\frac{1}{\zeta}} D^{-\frac{1}{\zeta^{2}}},\notag
	\end{align}
	then $g_n \to 0$ as $n \to +\infty$.
\end{lemma}

The following three results will be used in delicate iterations and comparison arguments with appropriate competitors. Moreover, they will be employed to obtain interior regularity of the gradients of  weak  solutions of  the  obstacle problem \eqref{main}.

\begin{lemma}[{\cite[Lemma A.3]{Zheng:2022}}]\label{lem4}Let $\tau$ be a non-negative function on an interval $(0, R_*]$ with $R_* \leq 1$. Suppose that for all $r$ and $R$ satisfying $ 0<r\leq R\leq R_*$ there holds
	\begin{align*}
		\tau(r) \leq A \left( \frac{r}{R} \right)^\alpha \tau(R) + BR^\beta, 
	\end{align*}
	where $A > 1,B, \alpha$, and $\beta$ are positive constants with $ \alpha>\beta$. Fix $\delta \in (\beta, \alpha)$ and consider $\theta \in (0, 1)$ with $A\theta^\alpha = \theta^\delta$. Suppose that there exists a constant $d > 0$ such that $\tau(r) \leq d\tau(\theta^k R)$ for all non-negative integer $k$ and $r \in [\theta^{k+1} R, \theta^k R]$. Then, there is a positive constant $C$ depending only on $ A,\alpha,\beta, \delta$, and $d$  such that
	\begin{align*}
		\tau(r) \leq C \left( \frac{r}{R} \right)^\sigma (\tau(R) + BR^\sigma)
	\end{align*}
	holds true for any $r,R$, and $\sigma$ satisfying $0<r\leq R\leq R_*$ and $0 < \sigma \leq \beta$. Furthermore, there is a positive constant $D  $ depending only on $B,C, \beta$, and $\tau(R_*)$  such that
	\begin{align*}
		\tau(r) \leq D r^\sigma
	\end{align*}
	holds true for any $r\in (0,R_*]$ and $\sigma\in (0,\beta]$.
\end{lemma}

\begin{lemma}[{\cite[Lemma 3.1]{Zheng:2017}}] \label{lem2}
	Let $u \in W^{1,p}(\Omega)$, $B_R \subset \Omega$. If $v$ is a bounded weak solution of the $p$-harmonic equation
	\begin{align}
		\mathrm{div} (|\nabla v|^{p-2}\nabla v) = 0 \text{ in } B_R, \quad v - u \in W_0^{1,p}(B_R),\notag
	\end{align}
	then for any $\mu \in (0, N)$, there exists a positive constant $C$ depending only on $N$, $p$, $\mu$,  and $\|v\|_{L^\infty(B_R)}$ such that
	\begin{align}
		\int_{B_R} |\nabla u - \nabla v|^p \mathrm{d}x \leq C \int_{B_R} (|\nabla u|^p - |\nabla v|^p) \mathrm{d}x + CR^{\frac{\mu}{2}} \left( \int_{B_R} (|\nabla u|^p - |\nabla v|^p) \mathrm{d}x \right)^{\frac{1}{2}}.\notag
	\end{align}
\end{lemma}

\begin{lemma}[{\cite[Lemma 4.1]{Leitao:2015}}] \label{lem3}
	Let $u \in W^{1,p}(B_R)$, $B_R \subset \Omega$. If $v \in W^{1,p}(B_R)$ is a weak solution the $p$-harmonic equation
	\begin{align}
		\mathrm{div} (|\nabla v|^{p-2}\nabla v) = 0 \text{ in } B_R,\notag
	\end{align}
	then there exist a constant $\delta \in (0, 1)$ and a positive constant $C$ depending only on $N$ and $p$ such that
	\begin{align}
		\int_{B_r} |\nabla u - (\nabla u)_r|^p \mathrm{d}x \leq C \left( \frac{r}{R} \right)^{N+\delta} \int_{B_R} |\nabla u - (\nabla u)_R|^p \mathrm{d}x + C \int_{B_R} |\nabla u - \nabla v|^p \mathrm{d}x,\forall r \in (0, R]. \notag
	\end{align}
	
\end{lemma}
%

\section{Existence  and uniform boundedness  of solutions}\label{Sec.3}

In this section, we   establish the existence   and uniform $L^\infty$ boundedness of solutions  to the  obstacle problem \eqref{main}.

The existence result is stated as below.
\begin{proposition}\label{exi}
	The equation \eqref{main} admits at least one  weak solution.
\end{proposition}
\begin{proof}  We employ the penalty method and the mountain pass lemma (see \cite{Ambrosetti:1973}) to prove this result. More specifically,
	we consider the associated  penalized problem, namely,  for $0 < \varepsilon < 1$, we consider the equation
	\begin{align} \label{main1}
		\begin{cases}
			\operatorname{div}\left( a(x)|\nabla u|^{p-2}\nabla u\right)
			= m_1 \chi_\varepsilon(u) - m_2 (u^{+})^{\lambda-1}
			\text{ in } \Omega, \\
			u = 0  \text{ on } \partial\Omega,
		\end{cases}
	\end{align}
	where $\chi_\varepsilon(s)$ is a smooth monotone approximation of the Heaviside function $\chi_{\{s>0\}}$ satisfying
	\begin{align}
		\chi_\varepsilon'(s) \geq 0, \quad \chi_\varepsilon(s) = 0 \text{ for } s \leq 0, \quad \chi_\varepsilon(s) = 1 \text{ for } s \geq\varepsilon.\notag
	\end{align}
	We claim that    a critical point, denoted by $ u_\varepsilon $, of the functional
	\begin{align}\label{mainb1}
		\mathcal{J}_{\varepsilon}(u) = \int_{\Omega}  \frac{a(x)}{p}|\nabla u|^p
		+ m_1 \Phi_{\varepsilon}(u)  - \frac{m_2}{\lambda}(u^+)^{\lambda}
		\mathrm{d}x
	\end{align}
	is a solution of equation \eqref{main1}, where $\Phi_\varepsilon(s) = \int_{-\infty}^s \chi_\varepsilon(t)\mathrm{d}t \geq 0$.
	
	Indeed, noting  that $\mathcal{J}_\varepsilon \in C^1(W_0^{1,p}(\Omega), \mathbb{R})$, it suffices to show that for sufficiently small $\varepsilon > 0$, the functional $\mathcal{J}_\varepsilon(u)$ admits a mountain pass geometry.
	
	First, it is clear that $\mathcal{J}_\varepsilon(0) = 0$. {Let  $S$ be  the optimal embedding constant of $W^{1,p}_0(\Omega)\hookrightarrow L^{\lambda}(\Omega)$.} For any $u\in W_0^{1,p}(\Omega)$ satisfying $\left( \int_{\Omega} |\nabla u|^p  \mathrm{d}x \right)^{\frac{1}{p}} = \left( \frac{a_0}{m_2} \right)^{\frac{1}{\lambda - p}} S^{\frac{\lambda}{\lambda - p}}$, by the Sobolev embedding theorem, we deduce that
	\begin{align}
		\mathcal{J}_{\varepsilon}(u) \geq& \int_{\Omega} \frac{a_0}{p}|\nabla u|^p -\frac{m_2}{\lambda}(u^+)^{\lambda}  \mathrm{d}x
		\notag\\ \geq
		&\frac{a_0}{p}\int_{\Omega} |\nabla u|^p \mathrm{d}x -\frac{m_2}{\lambda S^\lambda} \left( \int_{\Omega} |\nabla u|^p \mathrm{d}x\right)^\frac{\lambda}{p}
		\notag\\ =
		&\frac{1}{p}a_0^{\frac{\lambda}{\lambda-p}}m_2^{\frac{p}{p-\lambda}}S^{\frac{\lambda p}{\lambda-p}} - \frac{1}{\lambda}a_0^{\frac{\lambda}{\lambda-p}}m_2^{\frac{p}{p-\lambda}}S^{\frac{\lambda p}{\lambda-p}}
		\notag\\ =
		& \frac{\lambda-p}{\lambda p}a_0^{\frac{\lambda}{\lambda-p}}m_2^{\frac{p}{p-\lambda}}S^{\frac{\lambda p}{\lambda-p}}.\notag
	\end{align} 	
	
	Next, using the properties of $\chi_\varepsilon(s)$, we obtain
	\begin{align}
		\mathcal{J}_{\varepsilon}(tu) = &\int_{\Omega} \frac{a(x)t^p}{p}|\nabla u|^p + m_1 \Phi_{\varepsilon}(tu) -\frac{m_2t^\lambda}{\lambda}(u^+)^{\lambda}  \mathrm{d}x
		\notag\\ \leq
		&\frac{a_1t^p}{p}\int_{\Omega} |\nabla u|^p \mathrm{d}x + m_1\int_{\Omega}( |tu|+\varepsilon)\mathrm{d}x -\frac{m_2t^\lambda}{\lambda}\int_{\Omega}(u^+)^{\lambda}  \mathrm{d}x, \forall t>0, \label{1.6}
	\end{align}
	which, along with  $   \lambda  >p$,     implies that $\mathcal{J}(t u) \to -\infty$ as $t \to +\infty$.
	
	Therefore, the mountain pass lemma ensures that there exists a Palais-Smale  sequence for $	\mathcal{J}_{\varepsilon}(u)$ at the level {$c \geq  \frac{\lambda-p}{\lambda p}a_0^{\frac{\lambda}{\lambda-p}}m_2^{\frac{p}{p-\lambda}}S^{\frac{\lambda p}{\lambda-p}}  $.}

	In addition, we need to verify the Palais-Smale compactness condition. Suppose that there exists a Palais-Smale  sequence $\{u_n\}$ {such that $\mathcal{J}_\varepsilon(u_n)$} is bounded, i.e., $|\mathcal{J}_\varepsilon(u_n)| \leq M$ with some constant $M>0$, and $\mathcal{J}_\varepsilon'(u_n) \to 0$ in $\left( W^{1,p}_0(\Omega)\right)^*$. We first prove that $\{u_n\}$ is bounded in $W^{1,p}_0(\Omega)$,  and then  show the existence of a strongly convergent subsequence. Indeed, from \eqref{mainb1}, we obtain
	\begin{align}\label{1.8}
		\mathcal{J}_{\varepsilon}(u_n) = \int_{\Omega} \frac{a(x)}{p}|\nabla u_n|^p + m_1 \Phi_{\varepsilon}(u_n) -\frac{m_2}{\lambda}(u_n^+)^{\lambda}  \mathrm{d}x.
	\end{align}
	From \eqref{main1}, we obtain
	\begin{align}\label{1.9}
		\frac{1}{\lambda} \left\langle \mathcal{J}_\varepsilon'(u_n),u_n \right\rangle =  \int_{\Omega} \frac{a(x)}{\lambda}|\nabla u_n|^p +\frac{m_1}{\lambda}\chi_\varepsilon(u_n)u_n-\frac{m_2}{\lambda}(u_n^+)^{\lambda}  \mathrm{d}x.
	\end{align}
	Combining \eqref{1.8} and \eqref{1.9} yields
	\begin{align}
		\mathcal{J}_\varepsilon(u_n) - \frac{1}{\lambda} \langle \mathcal{J}_\varepsilon'(u_n), u_n \rangle
		\geq
		&\int_\Omega \left( \frac{a(x)}{p} - \frac{a(x)}{\lambda} \right)|\nabla u_n|^p \mathrm{d}x - \int_\Omega  \frac{m_1}{\lambda} \chi_\varepsilon(u_n)u_n  \mathrm{d}x
		+ \int_\Omega m_1 \Phi_\varepsilon(u_n)  \mathrm{d}x
		\notag\\\geq
		& \int_\Omega \left( \frac{a_0}{p} - \frac{a_0}{\lambda} \right)|\nabla u_n|^p \mathrm{d}x - \int_\Omega \frac{m_1}{\lambda} \chi_\varepsilon(u_n)u_n  \mathrm{d}x
		\notag \\\geq
		& \left( \frac{a_0}{p} - \frac{a_0}{\lambda} \right)\int_\Omega |\nabla u_n|^p \mathrm{d}x -C\int_\Omega |u_n| \mathrm{d}x , \label{1.10}
	\end{align} 	
	where    $C$ is a positive  constant depending only on $N$, $\lambda$, $m_1$, and $\Omega$.
	Since $\mathcal{J}_\varepsilon'(u_n) \to 0 $, for any $\varphi \in W^{1,p}_{0}(\Omega) $, it holds that
	\begin{align}
		|\langle \mathcal{J}_\varepsilon'(u_n), \varphi \rangle|\leq & \|\mathcal{J}_\varepsilon'(u_n)\|_{\left( W^{1,p}_0(\Omega)\right)^*} \|\varphi\|_{W^{1,p}_{0}(\Omega)}
		=
		o(1)\|\varphi\|_{W^{1,p}_{0}(\Omega)}. \label{1.11}
	\end{align}
	Substituting $\varphi = u_n$ into \eqref{1.11}, we get
	\begin{align}\label{1.12}
		\frac{1}{\lambda} |\langle \mathcal{J}_\varepsilon'(u_n), u_n \rangle| \leq o(1) \|u_n\|_{W^{1,p}_{0}(\Omega)}.
	\end{align}
	Using \eqref{1.12} and the condition $|\mathcal{J}_\varepsilon(u_n)| \leq M$, we infer that
	\begin{align}
		\mathcal{J}_\varepsilon(u_n) -	\frac{1}{\lambda} \langle \mathcal{J}_\varepsilon'(u_n), u_n \rangle \leq&  |\mathcal{J}_\varepsilon(u_n)| + \frac{1}{\lambda}|\langle \mathcal{J}_\varepsilon'(u_n), u_n \rangle|
		\notag\\\leq
		& o(1) \|u_n\|_{W^{1,p}_{0}(\Omega)}+M. \label{1.13}
	\end{align}
	Combining \eqref{1.10}, \eqref{1.13}, and Young's inequality with $\varepsilon_1,\varepsilon_2>0$, we obtain
	\begin{align}
		\left( \frac{a_0}{p} - \frac{a_0}{\lambda} \right)\int_\Omega |\nabla u_n|^p \mathrm{d}x \leq&  C\int_\Omega |u_n| \mathrm{d}x +o(1) \|u_n\|_{W^{1,p}_{0}(\Omega)}+ C
		\notag\\ \leq
		&  C\varepsilon_1\int_\Omega |u_n|^p \mathrm{d}x +o(1)\varepsilon_2\left( \int_\Omega |u_n|^p\mathrm{d}x + \int_\Omega |\nabla u_n|^p\mathrm{d}x\right) + C_{\varepsilon_1 ,\varepsilon_2}
		\notag\\ \leq
		&  C(\varepsilon_1+o(1)\varepsilon_2)\int_\Omega |\nabla u_n|^p \mathrm{d}x + C_{\varepsilon_1 ,\varepsilon_2}, \label{1.14}
	\end{align}
	where     $C$ is a  positive constant depending only on $N$, $p$, $\lambda$, $m_1$, $M$, and $\Omega$, and $C_{\varepsilon_1 ,\varepsilon_2}$ is a  positive constant depending additionally on $\varepsilon_1$ and $\varepsilon_2$.
	By choosing sufficiently small $\varepsilon_1,\varepsilon_2$  and using \eqref{1.14}, we obtain
	\begin{align}
		\int_\Omega |\nabla u_n|^p \mathrm{d}x \leq C,\notag
	\end{align}
	where     $C$ is a  positive constant depending only on  $N$, $p$, $\lambda$, $m_1$, $a_0$, $M$,    and $\Omega$.
	Therefore, the sequence $\{u_n\}$ is uniformly bounded in $W^{1,p}_0(\Omega)$. By the embedding theorem, there exist a subsequence, still denoted by $\{u_n\}$, and a function $u_\varepsilon \in W^{1,p}_0(\Omega)$ such that
	\begin{align}\label{1.15}
		u_n \rightharpoonup u_\varepsilon \text{ in } W^{1,p}_0(\Omega) \ \ \text{and}\ \
		u_n \to u_\varepsilon \text{ in } L^{p^*}(\Omega).
	\end{align}
	Since $\mathcal{J}_\varepsilon'(u_n) \to 0$, it holds that
	\begin{align} \label{1.16}
		\left| \langle \mathcal{J}_\varepsilon'(u_n), u_n - u_\varepsilon \rangle \right| \to 0.
	\end{align}
	By \eqref{1.16} and Lemma~\ref{Lemma6}, we deduce that
	\begin{align}
		0 \leftarrow &\left| \langle \mathcal{J}_\varepsilon'(u_n), u_n - u_\varepsilon \rangle \right|
		\notag \\
		=& \int_{\Omega}  a(x) |\nabla u_n|^{p-2} \nabla u_n (\nabla u_n - \nabla u_\varepsilon) + m_1 \chi_{\varepsilon}(u_n)(u_n - u_\varepsilon) - m_2 (u_n^+)^{\lambda - 2} u_n^+(u_n - u_\varepsilon)  \mathrm{d}x
		\notag\\=
		& \int_{\Omega}  a(x) \left( |\nabla u_n|^{p-2} \nabla u_n -  |\nabla u_\varepsilon|^{p-2} \nabla u_\varepsilon\right)  (\nabla u_n - \nabla u_\varepsilon)\mathrm{d}x
		\notag\\
		&+ \int_{\Omega}a(x)|\nabla u_\varepsilon|^{p-2} \nabla u_\varepsilon (\nabla u_n - \nabla u_\varepsilon)  \mathrm{d}x
		\notag\\
		&+  \int_{\Omega}m_1 \chi_{\varepsilon}(u_n)(u_n - u_\varepsilon) - m_2 (u_n^+)^{\lambda - 2} u_n^+(u_n - u_\varepsilon)  \mathrm{d}x
		\notag\\ \geq
		& \int_{\Omega} \left( \frac{1}{2}\right)^p a_0 |\nabla u_n - \nabla u_\varepsilon|^{p} \mathrm{d}x + \int_{\Omega}a(x)|\nabla u_\varepsilon|^{p-2} \nabla u_\varepsilon (\nabla u_n - \nabla u_\varepsilon)  \mathrm{d}x \notag\\
		&+  \int_{\Omega}m_1 \chi_{\varepsilon}(u_n)(u_n - u_\varepsilon) \mathrm{d}x - \int_{\Omega}m_2 (u_n^+)^{\lambda - 2} u_n^+(u_n - u_\varepsilon)  \mathrm{d}x.\label{1.17}
	\end{align}
	In view of \eqref{1.15} and the boundedness of $a(x)$, we obtain
	\begin{subequations}\label{1.18}
		\begin{align}
			\lim_{n \to \infty}\int_{\Omega}a(x)|\nabla u_\varepsilon|^{p-2} \nabla u_\varepsilon (\nabla u_n - \nabla u_\varepsilon)  \mathrm{d}x =&0 , \\
			\lim_{n \to \infty}\int_{\Omega}m_1 \chi_{\varepsilon}(u_n)(u_n - u_\varepsilon) \mathrm{d}x =&0 ,\\
			\lim_{n \to \infty}\int_{\Omega}m_2 (u_n^+)^{\lambda - 2} u_n^+(u_n - u_\varepsilon)  \mathrm{d}x=&0.
		\end{align}
	\end{subequations}
	From \eqref{1.17} and \eqref{1.18}, we infer that
	\begin{align}\label{1.19}
		\lim_{n \to \infty} \int_{\Omega}|\nabla u_n - \nabla u_\varepsilon|^{p} \mathrm{d}x=0,
	\end{align}
	which, along with \eqref{1.15}, implies that $u_n \to u_\varepsilon$ in $W^{1,p}_0(\Omega)$ as $n \to \infty$, hence the Palais-Smale condition is satisfied.
	
	We have verified that all conditions of the mountain pass lemma are satisfied. Therefore, there exists a nonzero critical point $u_\varepsilon$ such that $\mathcal{J}_\varepsilon'(u_\varepsilon) = 0$, and the corresponding critical value can be expressed as
	\begin{align}\label{1.20}
		\mathcal{J}_\varepsilon(u_\varepsilon)=c_\varepsilon: = \inf_{\gamma \in \Gamma} \max_{t \in [0,1]} \mathcal{J}_\varepsilon(\gamma(t))
	\end{align}
	with
	\begin{align}
		\Gamma := \left\{ \gamma \in C\left([0,1], W_0^{1,p}(\Omega)\right) : \gamma(0) = 0,\ \mathcal{J}_\varepsilon(\gamma(1)) < 0 \right\}.\notag
	\end{align}
	Therefore, $u_\varepsilon$ is a mountain pass solution of equation \eqref{main1} at the level $c_\varepsilon$.
	
	Now, we prove  the existence of a weak solution of equation \eqref{main}. Indeed, in view of \eqref{1.6}, for a fixed $u_0 \in W^{1,p}_0(\Omega) \setminus \{0\}$, there exists $t_0 > 0$ independent of $\varepsilon$ such that $\mathcal{J}_\varepsilon(t_0 u_0) < 0$. Setting $w_0 = t_0 u_0$, we have $t w_0 \in \Gamma$ for  $t \in [0,1]$, and there exists a constant $B (p, \lambda,m_1,m_2,a_1,w_0,\Omega)> 0$ independent of $\varepsilon$ such that
	\begin{align}
		c_\varepsilon \leq &\max_{t \in [0,1]} J_\varepsilon(tw_0)
		\notag\\=
		& \max_{t \in [0,1]} \left[  \int_\Omega \frac{a(x)t^p}{p}|\nabla w_0|^p \mathrm{d}x + m_1\int_\Omega \Phi_\varepsilon(tw_0) \mathrm{d}x -  \frac{m_2t^\lambda}{\lambda}\int_\Omega |w_0^+|^\lambda \mathrm{d}x \right]
		\notag\\\leq
		& \max_{t \in [0,1]} \left[ \frac{a_1t^p}{p}\int_\Omega |\nabla w_0|^p \mathrm{d}x + m_1\int_\Omega (|tw_0| + 1) \mathrm{d}x +\frac{m_2t^\lambda}{\lambda}\int_\Omega |w_0^+|^\lambda \mathrm{d}x  \right]
		\notag\\=
		&: B(p, \lambda,m_1,m_2,a_1,w_0,\Omega).\label{1.21}
	\end{align} 	
	From \eqref{1.10}, \eqref{1.20}, and \eqref{1.21}, we infer that
	\begin{align}
		B(p, \lambda,m_1,m_2,a_1,w_0,\Omega) \geq & c_\varepsilon
		\notag\\=
		&\mathcal{J}_\varepsilon(u_\varepsilon) - \frac{1}{\lambda} \langle \mathcal{J}_\varepsilon'(u_\varepsilon), u_\varepsilon \rangle
		\notag\\\geq
		& \left( \frac{a_0}{p} - \frac{a_0}{\lambda} \right)\int_\Omega |\nabla u_\varepsilon|^p \mathrm{d}x -C\int_\Omega |u_\varepsilon| \mathrm{d}x, \notag
	\end{align} 	
	which, along with  Young's inequality with $\varepsilon_1>0$, implies that
	\begin{align}
		\left( \frac{a_0}{p} - \frac{a_0}{\lambda} \right)\int_\Omega |\nabla u_\varepsilon|^p \mathrm{d}x \leq& C\int_\Omega |u_\varepsilon| \mathrm{d}x + C
		\leq
		C\varepsilon_1\int_\Omega |u_\varepsilon|^p \mathrm{d}x + C_{\varepsilon_1}.\label{1.23}
	\end{align}
	Choosing sufficiently small $\varepsilon_1$, we deduce from \eqref{1.23} that
	\begin{align}
		\int_\Omega |\nabla u_\varepsilon|^p \mathrm{d}x &\leq C,\notag
	\end{align}
	where  $C$ is a positive  constant depending only on $p$, $\lambda$, $m_1$, $m_2$,  $a_0$, $a_1$, $w_0$, and $\Omega$.
	Since the family $\{u_\varepsilon\}$ is uniformly bounded in $W^{1,p}_0(\Omega)$ with respect to $\varepsilon$, there exist a subsequence, still denoted by $\{u_{\varepsilon}\}$, and a function $u \in W^{1,p}_0(\Omega)$ such that
	\begin{align}
			u_{\varepsilon} \rightharpoonup u \text{ in } W^{1,p}_0(\Omega)\ \ \text{and}\ \
			u_{\varepsilon} \to u \text{ in } L^{p^*}(\Omega).\notag
		\end{align}
		Then, proceeding in the same way as for \eqref{1.19}, we obtain
		\begin{align}
			\lim_{m \to \infty} \int_{\Omega}|\nabla u_{\varepsilon} - \nabla u|^{p} \mathrm{d}x=0.\notag
		\end{align}
		Thus, $u_{\varepsilon} \to u $ in $ W^{1,p}_0(\Omega)$.
		Now, Taking $\varphi \in C_{0}^{\infty}(\Omega)$ as a test function for \eqref{main1}, we obtain
		\begin{align}\label{1.29}
			\int_{\Omega} a(x)|\nabla u_{u_{\varepsilon}}|^{p-2}\nabla u_{\varepsilon} \nabla\varphi  + m_1 \chi_{\varepsilon}(u_{\varepsilon})\varphi- m_2 (u_{\varepsilon}^{+})^{\lambda-1}\varphi\mathrm{d}x=0.
		\end{align}
		By virtue of $u_{\varepsilon} \to u$ in $W^{1,p}_0(\Omega)$  and  the subcritical growth condition \eqref{p-condition}, taking the limit $\varepsilon \to 0$ in \eqref{1.29}, we obtain
		\begin{align}
			\int_{\Omega} a(x)|\nabla u|^{p-2}\nabla u \nabla\varphi + m_1 \chi_{\{u>0\}}\varphi- m_2 (u^{+})^{\lambda-1}\varphi\mathrm{d}x=0,\notag
		\end{align}
		which, along with density of $C_{0}^{\infty}(\Omega)$ in $W^{1,p}_0(\Omega) $, implies that  $u$ satisfies the first equation of \eqref{main} in the weak sense. 
		Furthermore, one can show that such $u$ is non-negative,  and hence $u$ is  a weak solution of equation \eqref{main}. Indeed, using $u-u^+\in W^{1,p}_0(\Omega)$ as  a test  function for equation   \eqref{main}, we obtain
		\begin{align*}
			\int_{\Omega}a|\nabla u|^{p-2}\nabla u (\nabla u -\nabla u^+)\mathrm{d}x + \int_{\Omega} (m_1 \chi_{\{u>0\}} - m_2 u^{\lambda - 1} \chi_{\{u > 0\}}  )(u-u^+) \mathrm{d}x =0.
		\end{align*}
		It follows that
		\begin{align*}
			\int_{\Omega}a|\nabla u|^{p } \mathrm{d}x=& \int_{\Omega}a|\nabla u|^{p-2}\nabla u  \nabla u^+ \mathrm{d}x -\int_{\Omega} (m_1 \chi_{\{u>0\}} - m_2 u^{\lambda - 1} \chi_{\{u > 0\}}  )(u-u^+) \mathrm{d}x \notag\\
			\leq & \frac{1}{p}\int_{\Omega}a|\nabla u|^{p }  \mathrm{d}x + \frac{p-1}{p}\int_{\Omega}a|\nabla u^+|^{p }  \mathrm{d}x
			-\int_{\Omega} (m_1 \chi_{\{u>0\}} - m_2 u^{\lambda - 1} \chi_{\{u > 0\}}  )(u-u^+) \mathrm{d}x,
		\end{align*}
		where in the last inequality we used the Young's inequality. Therefore, we have
		\begin{align*}
			\int_{\Omega}a|\nabla u|^{p } \mathrm{d}x=& \int_{\Omega}a|\nabla u|^{p-2}\nabla u  \nabla u^+ \mathrm{d}x -\int_{\Omega} (m_1 \chi_{\{u>0\}} - m_2 u^{\lambda - 1} \chi_{\{u > 0\}}  )(u-u^+) \mathrm{d}x \notag\\
			\leq &  \int_{\Omega}a|\nabla u^+|^{p }  \mathrm{d}x
			-\frac{p}{p-1}\int_{\Omega} (m_1 \chi_{\{u>0\}} - m_2 u^{\lambda - 1} \chi_{\{u > 0\}}  )(u-u^+) \mathrm{d}x \notag\\
			= &  \int_{\Omega \cap \{u>0\} }a|\nabla u^+|^{p }  \mathrm{d}x
			-\frac{p}{p-1}\int_{\Omega \cap \{u>0\}} (m_1 \chi_{\{u>0\}} - m_2 u^{\lambda - 1} \chi_{\{u > 0\}}  )(u-u^+) \mathrm{d}x\notag\\
			& +  \int_{\Omega \cap \{u<0\}}a|\nabla u^+|^{p }  \mathrm{d}x
			-\frac{p}{p-1}\int_{\Omega \cap \{u<0\}} (m_1 \chi_{\{u>0\}} - m_2 u^{\lambda - 1} \chi_{\{u > 0\}}  )(u-u^+) \mathrm{d}x \notag\\
			= &  \int_{\Omega \cap \{u>0\} }a|\nabla u|^{p }  \mathrm{d}x,
		\end{align*}
		which implies that
		\begin{align}
			\int_{\Omega \cap \{u<0\} }a|\nabla u|^{p }  \mathrm{d}x = 0.   \notag
		\end{align}
		Thus, either $u \equiv C$ or $u \geq 0$ a.e. in $\Omega$. Since $u \in W^{1,p}_0(\Omega)$, it follows that $u \geq 0$ a.e. in $\Omega$.
	\end{proof}
	
	The following result gives a uniform $L^{\infty}$ bound for the solutions of equation \eqref{main}.
	\begin{proposition}\label{exi-2}
		All weak solutions, still denoted by $u$, of equation \eqref{main} are uniformly bounded, having the estimate
		\begin{align*}
			\|u\|_{L^{\infty}(\Omega)}\leq C,
		\end{align*}
		where $C$ is positive constant  depending only on   $N$, $p$, $\lambda$, $m_1$, $m_2$,  $a_0$,  and $\Omega$.
	\end{proposition}
	\begin{proof}
		Fix $k_0 > 0$. For every $k \geq k_0$, define the function $u_k : \Omega \to \mathbb{R}$ by
		\begin{align}
			u_k(x) =
			\begin{cases}
				k  , & \text{if }  u(x)  > k, \\
				u(x), & \text{if }  u(x)  \leq k.
			\end{cases}\notag
		\end{align}
		Let $S_k = \{ x \in \Omega :  u(x)  > k \}$. Then,   for all $k \geq k_0$ we have
		\begin{align}
			u_k = u  \text{  in  } S_k^c  \text{  and  }  u_k = k   \text{  in  } S_k.\notag
		\end{align}
		Using $u-u_k\in W^{1,p}_0(\Omega)$ as a test function for equation   \eqref{main}, we obtain
		\begin{align*}
			- \int_{\Omega}a|\nabla u|^{p-2}\nabla u (\nabla u -\nabla u_k)\mathrm{d}x= \int_{\Omega} (m_1 \chi_{\{u>0\}} - m_2 u^{\lambda - 1} \chi_{\{u > 0\}}  )(u-u_k) \mathrm{d}x.
		\end{align*}
		It follows that
		\begin{align*}
			\int_{\Omega}a|\nabla u|^{p } \mathrm{d}x=& \int_{\Omega}a|\nabla u|^{p-2}\nabla u  \nabla u_k \mathrm{d}x -\int_{\Omega} (m_1 \chi_{\{u>0\}} - m_2 u^{\lambda - 1} \chi_{\{u > 0\}}  )(u-u_k) \mathrm{d}x \notag\\
			\leq & \frac{1}{p}\int_{\Omega}a|\nabla u|^{p }  \mathrm{d}x + \frac{p-1}{p}\int_{\Omega}a|\nabla u_k|^{p }  \mathrm{d}x
			-\int_{\Omega} (m_1 \chi_{\{u>0\}} - m_2 u^{\lambda - 1} \chi_{\{u > 0\}}  )(u-u_k) \mathrm{d}x,
		\end{align*}
		where in the last inequality we used the Young's inequality. Therefore, we have
		\begin{align*}
			\int_{\Omega}a|\nabla u|^{p } \mathrm{d}x=& \int_{\Omega}a|\nabla u|^{p-2}\nabla u  \nabla u_k \mathrm{d}x -\int_{\Omega} (m_1 \chi_{\{u>0\}} - m_2 u^{\lambda - 1} \chi_{\{u > 0\}}  )(u-u_k) \mathrm{d}x \notag\\
			\leq &  \int_{\Omega}a|\nabla u_k|^{p }  \mathrm{d}x
			-\frac{p}{p-1}\int_{\Omega} (m_1 \chi_{\{u>0\}} - m_2 u^{\lambda - 1} \chi_{\{u > 0\}}  )(u-u_k) \mathrm{d}x,
		\end{align*}
		which, along with the non-negativity of solutions, implies that
		\begin{align}
			\int_{S_k}a_0|\nabla u|^{p } \mathrm{d}x \leq &  \int_{S_k}a |\nabla u|^{p } \mathrm{d}x\notag\\
			=& \int_{\Omega}a|\nabla u|^{p } \mathrm{d}x- \int_{\Omega}a|\nabla u_k|^{p }  \mathrm{d}x
			\notag\\
			\leq &
			-\frac{p}{p-1}\int_{\Omega} (m_1 \chi_{\{u>0\}} - m_2 u^{\lambda - 1} \chi_{\{u > 0\}}  )(u-u_k) \mathrm{d}x\notag\\
			= &
			-\frac{p}{p-1}\int_{S_k} (m_1 \chi_{\{u>0\}} - m_2 u^{\lambda - 1} \chi_{\{u > 0\}}  )(u-u_k) \mathrm{d}x\notag\\
			=&-\frac{p}{p-1}\int_{S_k} (m_1   - m_2 u^{\lambda - 1}   )(u-k) \mathrm{d}x\notag\\
			\leq &\frac{p}{p-1}\int_{S_k}m_2 u^{\lambda - 1}    (u-k) \mathrm{d}x\notag\\
			\leq &\frac{p}{p-1}\int_{S_k}m_2 u^{\lambda  }    \mathrm{d}x\notag\\
			\leq
			&  \frac{p m_2}{p-1} 2^{\lambda-1}\int_{S_k}  \left( u  - k  \right)^\lambda  +k^\lambda \mathrm{d}x
			\notag\\\leq
			& \frac{p m_2}{p-1} 2^{\lambda-1} \int_{S_k}  \left| u  - k  \right|^\lambda   \mathrm{d}x + \frac{p m_2}{p-1} 2^{\lambda-1} k^\lambda |S_k|.   \notag
		\end{align}
		Thus, it holds that
		\begin{align}\label{3.3}
			\int_{S_k} |\nabla u|^p  \mathrm{d}x \leq & C\int_{S_k}  \left( u - k  \right)^\lambda   \mathrm{d}x + Ck^\lambda |S_k|,
		\end{align}
		where   $C$ is a positive  constant  depending only on $p$, $\lambda$, $m_2$,  and $a_0$.
		
		Define $K_n := \frac{K}{2} \left(1 -  \frac{1}{2^{n+1}}\right)$ for all $n \in \mathbb{N}$, where $K_n, K \geq k_0$. Set $g_n := \int_{S_{K_n}} (( u  - K_n)^+)^\lambda \mathrm{d}x$ for all $n \in \mathbb{N}$. We claim that
		\begin{align}\label{3.4}
			g_{n+1} \leq C D^n g_n^{1+\zeta}, {\forall n \in \mathbb{N},}
		\end{align}
		where $C > 0$, $\zeta > 0$, and $D > 1$ are constants independent of $n$.

		Let \begin{align*}
			q=
			\begin{cases}
				p^*,&N \geq 3, \\
				p+\lambda,& N = 2.
			\end{cases}
		\end{align*}
		By H\"older's inequality and the embedding inequality, it follows that
		\begin{align}
			g_{n+1}
			=
			&\int_{S_{K_{n+1}}} \left((u- K_{n+1})^+\right)^\lambda  \mathrm{d}x
			\notag \\ \leq
			& \left( \int_{S_{K_{n+1}}}  \left((u - K_{n+1})^+\right)^{q}  \mathrm{d}x \right)^{\frac{\lambda}{q}} |S_{K_{n+1}}|^{1-\frac{\lambda}{q}}
			\notag \\\leq
			& \left( \int_{\Omega}  \left((u - K_{n+1})^+\right)^{q}  \mathrm{d}x \right)^{\frac{\lambda}{q}} |S_{K_{n+1}}|^{1-\frac{\lambda}{q}}
			\notag \\ \leq
			& C\left( \int_{\Omega} \left| \nabla (u - K_{n+1})^+\right| ^p  \mathrm{d}x \right)^{\frac{\lambda}{p}} |S_{K_{n+1}}|^{1-\frac{\lambda}{q}}
			\notag \\ \leq
			& C \left( \int_{S_{K_{n+1}}} |\nabla u|^p  \mathrm{d}x \right)^{\frac{\lambda}{p}} |S_{K_{n+1}}|^{1-\frac{\lambda}{q}},\label{3.5}
		\end{align}
		where $C$ is a positive  constant  depending only on $N$, $p$,  $\lambda$, and $\Omega$.
		
		Noting that $ K_{n+1} - K_n = \frac{K}{2^{n+3}} $, we obtain
		\begin{align}
			\left( \frac{K}{2^{n+3}} \right)^\lambda |S_{K_{n+1}}|
			=
			& (K_{n+1} - K_n)^\lambda |S_{K_{n+1}}|
			\notag\\ =
			& \int_{S_{K_{n+1}}} |K_{n+1} - K_n|^\lambda    \mathrm{d}x
			\notag\\ \leq
			&\int_{S_{K_{n+1}}} \left((u - K_n)^+\right)^\lambda    \mathrm{d}x
			\notag\\ \leq
			& \int_{S_{K_{n}}} \left((u - K_n)^+\right)^\lambda    \mathrm{d}x
			\notag\\ =
			& g_n,\notag
		\end{align}
		which implies that
		\begin{align}\label{3.7}
			|S_{K_{n+1}}| \leq \left( \frac{2^{n+3}}{K} \right)^\lambda g_n.
		\end{align}
		
		Using \eqref{3.7}, we deduce that
		\begin{align}
			\int_{S_{K_{n+1}}} \left| u - K_{n+1} \right|^\lambda \mathrm{d}x
			\leq
			& 2^\lambda \int_{S_{K_{n+1}}} \left| u- K_n \right|^\lambda \mathrm{d}x + 2^\lambda \int_{S_{K_{n+1}}} |K_n - K_{n+1}|^\lambda \mathrm{d}x
			\notag\\ \leq
			& 2^\lambda \int_{S_{K_n}} \left|u - K_n \right|^\lambda \mathrm{d}x + 2^\lambda |K_n - K_{n+1}|^\lambda |S_{K_{n+1}}|
			\notag\\ \leq
			& 2^\lambda g_n + 2^\lambda \left(\frac{2^{n+3}}{K}\right)^\lambda \left(\frac{K}{2^{n+3}}\right)^\lambda g_n
			\notag\\=
			& 2^{\lambda+1}g_n. \label{3.8}
		\end{align}
		
		Then we have
		\begin{align}\label{3.9}
			K_{n+1}^\lambda |S_{K_{n+1}}| \leq \left(\frac{K}{2} \left(1 - \frac{1}{2^{n+2}}\right)\right)^\lambda \left(\frac{2^{n+3}}{K}\right)^\lambda g_n = \left(\frac{2^{n+3}}{2} \left(1 - \frac{1}{2^{n+2}}\right)\right)^\lambda g_n \leq 2^{2n+4} g_n.
		\end{align}
		
		Combining  \eqref{3.3}, \eqref{3.5}, \eqref{3.7}, \eqref{3.8}, and \eqref{3.9}, we obtain
		\begin{align}
			g_{n+1}
			\leq
			& C \left( \int_{S_{K_{n+1}}} \left|u- K_{n+1} \right|^\lambda dx + K_{n+1}^\lambda |S_{K_{n+1}}| \right)^{\frac{\lambda}{p}} |S_{K_{n+1}}|^{1-\frac{\lambda}{q}}
			\notag \\ \leq
			& C \left( 2^{\lambda+1}g_n + 2^{\lambda (n+2)} g_n \right)^{\frac{\lambda}{p}} \left( \left(\frac{2^{n+3}}{K}\right)^\lambda g_n\right)^{1-\frac{\lambda}{q}}
			\notag \\ \leq
			& C \left(  2^{\lambda+1} + 2^{\lambda (n+2)} \right)^{\frac{\lambda}{p}} g_n^{\frac{\lambda}{p}} \left( \frac{2^{n+3}}{K}\right)^{\lambda-\frac{\lambda^2}{q}} g_n^{1-\frac{\lambda}{q}}
			\notag \\ \leq
			& C \left( 2^{\lambda (n+2)}\right)^{\frac{\lambda}{p}} \left(2^{n+3}\right)^{\lambda-\frac{\lambda^2}{q}} g_n^{1+{\frac{\lambda}{p}}-\frac{\lambda}{q}}
			\notag \\ \leq
			& C2^{\frac{\lambda^2 n}{p}+\frac{2\lambda^2 }{p}} 2^{n\lambda-\frac{n\lambda^2}{q}} g_n^{1+{\frac{\lambda}{p}}-\frac{\lambda}{q}}
			\notag \\ \leq
			& C 2^n 2^{n\lambda\left(1-\frac{\lambda}{q}\right)}g_n^{1+{\frac{\lambda}{p}}-\frac{\lambda}{q}}
			\notag \\  =
			&C\left(2^{ 1+\lambda }\right)^n g_n^{1+\zeta},\notag
		\end{align}
		where $\zeta := \frac{\lambda}{p} - \frac{\lambda}{q} > 0$.
		Therefore, inequality \eqref{3.4} holds with $D := 2^{\lambda+1}$ and some constant $C := C(N, p, \lambda,  m_2,  a_0, K, \Omega)$. In particular, we have
		\begin{align}
			\left(\frac{K}{4}\right)^{q} |S_{K_0}| = (K_0)^{q} |S_{K_0}| \leq \int_{S_{K_0}} |u|^{q}  \mathrm{d}x  \leq \int_{\Omega} |u|^{q}  \mathrm{d}x =  \|u\|_{L^{q}(\Omega)}^{q}.\label{3.10}
		\end{align}
		
		From \eqref{3.10} and the embedding theorem, we infer that
		\begin{align}
			g_0 = \int_{S_{K_0}} \left((u- K_0)^+\right)^\lambda \mathrm{d}x \leq \int_{S_{K_0}} |u|^\lambda \mathrm{d}x \leq \|u\|_{L^{q}(\Omega)}^\lambda |S_{K_0}|^{1-\frac{\lambda}{q}} \leq \left(\frac{4}{K}\right)^{q-\lambda} \|u\|_{L^{q}(\Omega)}^{q}.\notag
		\end{align}
		
		Therefore, by choosing a sufficiently large $K$ depending only on $\|u\|_{L^{q}(\Omega)}$, we obtain
		\begin{align}
			g_0
			\leq
			\left(\frac{4}{K}\right)^{q-\lambda} \|u\|_{L^{q}(\Omega)}^{q}
			\leq
			C^{- \frac{1}{\zeta}} D^{- \frac{1}{\zeta^2}}.	\label{3.12}
		\end{align}
		
		In view of \eqref{3.4} and \eqref{3.12}, we deduce by  Lemma~\ref{lem1}  that $g_n \to 0$ as $n \to \infty$, and therefore
		\begin{align}
			\int_{S_{K/2}} \left( \left( u - \frac{K}{2} \right)^+ \right)^p   \mathrm{d}x= \lim_{n \to +\infty} \int_{S_{K_n}} \left( (u - K_n)^+ \right)^p  \mathrm{d}x=0.\notag
		\end{align}
		Thus,
		\begin{align}\label{3.32}
			0\leq u \leq \frac{K}{2} \ \     \text{a.e.\  in}\ \  \Omega,
		\end{align}
		where $K$ is a positive constant depending on $\| u \|_{L^{q}(\Omega)}$.
		
		We now estimate   $\| u \|_{L^{q}(\Omega)}$. Indeed, by the embedding theorem, we obtain
		\begin{align}\label{3.13}
			\| u \|_{L^{q}(\Omega)}\leq C \| \nabla u \|_{L^{p}(\Omega)},
		\end{align}
		where $C$ is a positive  constant  depending only on $N$, $p$, $q$, and $\Omega$, and hence, only on $N$, $p$, $\lambda$, and $\Omega$.
		Furthermore, taking  $u$ as test function for equation \eqref{main}, then applying Young's inequality with $\varepsilon_1,\varepsilon_2 >0$, the embedding inequality, and the Poincar\'e's inequality, we obtain
		\begin{align}
			\int_{\Omega} a_0  | \nabla u |^p \mathrm{d}x
			\leq
			& \int_{\Omega} a(x)  | \nabla u |^p \mathrm{d}x
			\notag\\ =
			&-\int_{\Omega} m_1 \chi_{\{u>0\}}\ u \mathrm{d}x+ \int_{\Omega} m_2 u^{\lambda - 1} \chi_{\{u > 0\}} u \mathrm{d}x
			\notag\\ \leq
			& \int_{\Omega} m_1 | u| \mathrm{d}x+ \int_{\Omega} m_2 |u|^{\lambda}  \mathrm{d}x
			\notag\\ \leq
			&  \varepsilon_1\int_{\Omega} | u|^p \mathrm{d}x+ \varepsilon_2 \int_{\Omega} |u|^{q}  \mathrm{d}x +C_{\varepsilon_1, \varepsilon_2}
			\notag\\ \leq
			&  C \varepsilon_1\int_{\Omega} | \nabla u|^p \mathrm{d}x+ C\varepsilon_2 \int_{\Omega} |\nabla u|^{p}  \mathrm{d}x +C_{\varepsilon_1, \varepsilon_2},\label{3.14}
		\end{align}
		where   $C$ is  a positive constant depending only on $p$, $N$, and $\Omega$, while $C_{\varepsilon_1, \varepsilon_2}$ is  a positive constant depending  only on $\varepsilon_1$, $\varepsilon_2$, $N$, $p$, $\lambda$, $m_1$, $m_2$,    and $\Omega$.
		Let  $\varepsilon_1$ and $\varepsilon_2$ be  sufficiently small. It follows from \eqref{3.14} that
		\begin{align}\label{3.15}
			\int_{\Omega}  | \nabla u |^p \mathrm{d}x &\leq C,
		\end{align}
		where $C$ is a positive  constant  depending only on $N$, $p$, $\lambda$, $m_1$, $m_2$, $a_0$, and $\Omega$.
		
		From \eqref{3.32}, \eqref{3.13}, and \eqref{3.15}, we conclude  that all solutions have the same  bound, which {depends  only on} $N$, $p$, $\lambda$, $m_1$, $m_2$,  $a_0$,  and $\Omega$.
	\end{proof}

	\section{Uniformly local H\"older continuity of solutions and their gradients}\label{Sec.4}
	
	In the section, we establish uniformly local H\"older continuity of  solutions and their gradients for the obstacle problem \eqref{main}.
	

	First, we have the following result regarding the uniformly local H\"older continuity of  solutions.
	\begin{proposition}\label{C^0} All solutions, denoted by  $u$, of equation \eqref{main} belong to $C_{\mathrm{loc}}^{0, \alpha}(\Omega)$ with the same constant $\alpha \in (0, 1)$. Moreover,  for any $\Omega' \subset \subset \Omega$, there exists a positive constant $C$ depending on $N$, $p$, $\lambda$, $m_1$, $m_2$, $a_0$, $\mu$, $\Omega'$, and $\Omega$ such that
		\begin{align}
			\|u\|_{C^{0,\alpha}(\Omega')} \leq C.\notag
		\end{align}
	\end{proposition}
	
	\begin{proof}
		For any $x_0\in \Omega$, let  $B_R = B_R(x_0)\subset\subset \Omega$ with some $ R>0$.  Let $v$ be a $p$-harmonic function satisfying
		\begin{align}\label{4.511}
			\begin{cases}
				\mathrm{div}\left(|\nabla v|^{p-2} \nabla v\right) = 0  \text{ in } B_{R}, \\
				v = u  \text{ on } \partial B_{R}.
			\end{cases}
		\end{align}
		By the maximum principle, we obtain
		\begin{align}
			\|v\|_{L^\infty(B_R)} \leq \|v\|_{L^\infty(\partial B_R)} = \|u\|_{L^\infty(\partial B_R)} \leq \|u\|_{L^\infty(\Omega)}. \label{Equ4.2}
		\end{align}
		Using $u-v\in W^{1,p}_0(B_R)$ as  a test  function for equation   \eqref{main}, we obtain
		\begin{align*}
			\int_{B_R}a|\nabla u|^{p-2}\nabla u (\nabla u -\nabla v)\mathrm{d}x + \int_{B_R} (m_1 \chi_{\{u>0\}} - m_2 u^{\lambda - 1} \chi_{\{u > 0\}}  )(u-v) \mathrm{d}x =0.
		\end{align*}
		It follows that
		\begin{align*}
			\int_{B_R}a|\nabla u|^{p } \mathrm{d}x=& \int_{B_R}a|\nabla u|^{p-2}\nabla u  \nabla v \mathrm{d}x -\int_{B_R} (m_1 \chi_{\{u>0\}} - m_2 u^{\lambda - 1} \chi_{\{u > 0\}}  )(u-v) \mathrm{d}x \notag\\
			\leq & \frac{1}{p}\int_{B_R}a|\nabla u|^{p }  \mathrm{d}x + \frac{p-1}{p}\int_{B_R}a|\nabla v|^{p }  \mathrm{d}x
			-\int_{B_R} (m_1 \chi_{\{u>0\}} - m_2 u^{\lambda - 1} \chi_{\{u > 0\}}  )(u-v) \mathrm{d}x,
		\end{align*}
		where in the last inequality we used the Young's inequality. Therefore, we have
		\begin{align}
			\int_{B_R}a|\nabla u|^{p } \mathrm{d}x -  \int_{B_R}a|\nabla v|^{p }  \mathrm{d}x 
			\leq &
			-\frac{p}{p-1}\int_{B_R} (m_1 \chi_{\{u>0\}} - m_2 u^{\lambda - 1} \chi_{\{u > 0\}}  )(u-v) \mathrm{d}x \notag\\
			\leq & \frac{p}{p-1} \left( \int_{B_R } m_1|v - u| \mathrm{d}x + \int_{B_R } m_2\|u\|_{L^\infty(\Omega)}^{\lambda-1}|v-u| \mathrm{d}x\right)  \notag\\
			\leq & C \int_{B_R } |v - u|\mathrm{d}x  \label{4.5-1}
			\\ \leq
			& C R^N \label{4.5-2},
		\end{align}
		where, by virtue of \eqref{Equ4.2} and Proposition~\ref{exi-2},   $C$ is a positive constant depending only on  $N$, $p$, $\lambda$, $m_1$, $m_2$, and $\Omega$.

		By \eqref{4.5-2} and \cite[Lemma A.14]{Zheng:2022}, for any $\kappa \in (0, N)$, there holds
		\begin{align*}
			\int_{B_R}  |\nabla u|^p \mathrm{d}x \leq&
			\frac{a}{a_0}\int_{B_R}  |\nabla u|^p \mathrm{d}x
			\notag\\
			\leq &\frac{a}{a_0}\int_{B_R}  |\nabla v|^p \mathrm{d}x  + C R^N
			\notag\\
			\leq & \frac{a_1}{a_0}\int_{B_R}  |\nabla v|^p \mathrm{d}x  + C R^N
			\notag\\
			\leq& C R^\kappa ,
		\end{align*}
		where we let $\kappa \in(N - 1 , N)$. Applying \cite[Lemma A.13]{Zheng:2022}, we conclude that $ u \in C^{0,\alpha}_{\mathrm{loc}}(\Omega) $ for some $ \alpha \in (0,1)$.
	\end{proof}
	
	The local H\"older continuity  for the solutions' gradients is given below.
	\begin{theorem}\label{C^1}
		All solutions, denoted by  $u$, of equation \eqref{main} belong to $C_{\mathrm{loc}}^{1, \alpha}(\Omega)$ with the same constant $\alpha \in (0, 1)$.  Moreover, for any $\Omega' \subset \subset \Omega$, there exists a positive constant $C$ depending on $N$, $p$, $\lambda$, $m_1$, $m_2$, $a_0$, $\mu$, $\Omega'$, and $\Omega$ such that
		\begin{align}
			\|u\|_{C^{1,\alpha}(\Omega')} \leq C. \label{Equ 4.5}
		\end{align}
	\end{theorem}
	
	\begin{proof} For any $x_0\in \Omega$, let  $B_R = B_R(x_0)\subset\subset \Omega$ with some $ R \leq R_* \leq 1$.
		Let $ v $ be  a $p$-harmonic function satisfying \eqref{4.511}. For any $B_r \subset\subset B_R $,	by Lemma~\ref{lem2} and Lemma~\ref{lem3}, there exists a constant $\delta \in (0,1)$ such that
		\begin{align}
			\int_{B_r} |\nabla u - (\nabla u)_r|^p  \mathrm{d}x
			\leq
			& C \left( \frac{r}{R} \right)^{N+\delta} \int_{B_R} |\nabla u - (\nabla u)_R|^p  \mathrm{d}x + C\int_{B_R} |\nabla u - \nabla v|^p  \mathrm{d}x
			\notag \\ \leq
			& C\left( \frac{r}{R} \right)^{N+\delta} \int_{B_R} |\nabla u - (\nabla u)_R|^p \mathrm{d}x  + C\int_{B_R} \left(|\nabla u|^p - |\nabla v|^p\right) \mathrm{d}x
			\notag\\
			&+ C R^{\frac{\mu}{2}} \left( \int_{B_R} \left(|\nabla u|^p - |\nabla v|^p\right) \mathrm{d}x \right)^{\frac{1}{2}}, \label{4.2}
		\end{align}
		where $\mu \in (0, N)$ is an arbitrary constant, and  $C$ is  a  positive  constant depending only on $N$, $p$, $\mu$,  and $\|v\|_{L^\infty(B_R)}$.

		By \eqref{4.5-1}, we obtain
		\begin{align}\label{4.5}
			\int_{B_R} \left(  |\nabla u|^p -  |\nabla v|^p \right)  \mathrm{d}x \leq C \int_{B_R} |v- u|\mathrm{d}x,
		\end{align}
		where $C$ is a positive  constant  depending only on  $N$, $p$, $\lambda$, $m_1$, $m_2$, $a_0$, and $\Omega$.
		
		By Young's inequality with $\varepsilon_1>0$, we obtain
		\begin{align}
			\int_{B_R} |v - u| \mathrm{d}x
			\leq
			&|B_R|^{\frac{1}{N}}  \|v - u\|_{L^{1^*}(B_R)}
			\notag\\ \leq
			& C R \int_{B_R} |\nabla(v - u)| \mathrm{d}x
			\notag\\ \leq
			& C \varepsilon_1 R \int_{B_R} |\nabla v - \nabla u|^p \mathrm{d}x + C_{\varepsilon_1} R^{N+1},\label{4.6}
		\end{align}
		where $C$ is a positive  constant  depending only on $N$, $p$, and $\Omega$, and where $\varepsilon_1 > 0$ will be chosen later.
		
		We infer from Lemma~\ref{lem2}, \eqref{4.5}, and \eqref{4.6} that
		\begin{align}
			\int_{B_R} |\nabla u - \nabla v|^p \mathrm{d}x
			\leq
			& C \int_{B_R} \left(|\nabla u|^p - |\nabla v|^p\right) \mathrm{d}x + C R^{\frac{\mu}{2}} \left( \int_{B_R} \left(|\nabla u|^p - |\nabla v|^p\right) \mathrm{d}x \right)^{\frac{1}{2}}
			\notag\\ \leq
			&  C \varepsilon_1 R \int_{B_R} |\nabla u - \nabla v|^p \mathrm{d}x + C_{\varepsilon_1} R^{N+1}  + C \varepsilon_1^{\frac{1}{2}} R^{\frac{\mu+1}{2}} \left( \int_{B_R} |\nabla u - \nabla v|^p \mathrm{d}x \right)^{\frac{1}{2}} + C_{\varepsilon_1} R^{\frac{\mu+1+N}{2}}
			\notag\\ \leq
			&  C \varepsilon_1 R \int_{B_R} |\nabla u - \nabla v|^p\mathrm{d}x + C_{\varepsilon_1} R^{N+1} + \varepsilon_1 \int_{B_R} |\nabla u - \nabla v|^p \mathrm{d}x + C_{\varepsilon_1} R^{\frac{\mu+1+N}{2}} +  C R^{\mu+1}. \notag
		\end{align}
		Considering a sufficiently small $\varepsilon_1$, we obtain
		\begin{align}\label{4.7}
			\int_{B_R} |\nabla u - \nabla v|^p \mathrm{d}x \leq C  R^m,
		\end{align}
		where $m := \min \left\{ N + 1, \mu + 1, \frac{\mu + 1 + N}{2} \right\}$, and $C$ is a positive  constant  depending only on  $N$,  $p$, $\lambda$, $m_1$, $m_2$, $a_0$, $\mu$,  and  $\Omega$.

		From \eqref{4.5}, \eqref{4.6},  \eqref{4.7},  and the choice of $\varepsilon_1$, we obtain
		\begin{align}
			\int_{B_R} \left(|\nabla u|^p - |\nabla v|^p\right) \mathrm{d}x
			\leq
			& C \int_{B_R} |\nabla u - \nabla v|^p \mathrm{d}x + CR^{N+1}
			\notag\\ \leq
			& C R^{m+1} + C R^{N+1},\label{4.8}
		\end{align}
		where $C$ is a positive  constant  depending only on  $N$,  $p$, $\lambda$, $m_1$, $m_2$, $a_0$, $\mu$, and $\Omega$.

		Substituting \eqref{4.8} into \eqref{4.2}, we obtain that
		\begin{align}
			\int_{B_r} |\nabla u - (\nabla u)_r|^p \mathrm{d}x &\leq C \left( \frac{r}{R} \right)^{N+\delta} \int_{B_R} |\nabla u - (\nabla u)_R|^p \mathrm{d}x + C R^{N+1} + C R^{m+1}
			+ C R^{\frac{\mu+N+1}{2}} + C R^{\frac{\mu+m+1}{2}}, \forall r \in (0, R].\notag
		\end{align}
		Since $\mu \in (0, N)$ is arbitrary, we may choose $\mu \in (N-1, N)$, which then implies that $ \min \left\{ m + 1, \frac{\mu + N + 1}{2}, \frac{\mu + m + 1}{2} \right\} > N$. Consequently, there exists a constant $\alpha_1 > 0$ such that
		\begin{align}
			\int_{B_r} |\nabla u - (\nabla u)_r|^p \mathrm{d}x \leq C \left( \frac{r}{R} \right)^{N+\delta} \int_{B_R} |\nabla u - (\nabla u)_R|^p \mathrm{d}x + C R^{N+\alpha_1},\notag
		\end{align}
		where $C$ is a positive  constant  depending only on  $N$,  $p$, $\lambda$, $m_1$, $m_2$, $a_0$, $\mu$, and $\Omega$.
		
		By the inequality (5.1) in \cite{Lieberman:1991} it holds that
		\begin{align}
			\int_{B_r} |\nabla u - (\nabla u)_r|^p \mathrm{d}x
			\leq
			& C \int_{B_r} |\nabla u - (\nabla u)_R|^p \mathrm{d}x
			\notag\\ \leq
			&C \int_{B_R} |\nabla u - (\nabla u)_R|^p \mathrm{d}x,\notag
		\end{align}
		where $C$ is a positive  constant  depending only on $N$ and $p$.
		
		Applying Lemma~\ref{lem4}  with $\tau(r) := \int_{B_r} |\nabla u - (\nabla u)_r|^p\mathrm{d}x$, we find that there exists a constant $\alpha_2 \in (0, 1)$ such that
		\begin{align}
			\int_{B_r} |\nabla u - (\nabla u)_r|^p \mathrm{d}x \leq C r^{N+\alpha_2},\notag
		\end{align}
		where $C$ is a positive  constant  depending only on $N$,  $p$, $\lambda$, $m_1$, $m_2$, $a_0$, $\mu$, and $\Omega$.
		It follows that
		\begin{align}
			\int_{B_r} |\nabla u - (\nabla u)_r|\mathrm{d}x
			\leq
			&	\left( \int_{B_r} |\nabla u - (\nabla u)_r|^p \mathrm{d}x\right) ^\frac{1}{p} |B_r|^\frac{p-1}{p}
			\notag\\ \leq
			& Cr^{N+\alpha} \notag
		\end{align}
		with $\alpha := \frac{\alpha_2}{p} \in (0,1)$, where $C$ is  a positive constant depending only on $N$,  $p$, $\lambda$, $m_1$, $m_2$, $a_0$, $\mu$, and $\Omega$. Thus, $\nabla u$ belongs to the Campanato space. According to the Campanato's embedding theorem, we conclude that $u$ is locally  $C^{1,\alpha}$ continuous and \eqref{Equ 4.5} holds true.
	\end{proof}

	\section{Porosity of the free boundary}\label{Sec.5}
	This section  first presents some results on the optimal growth and  non-degeneracy of solutions of equation \eqref{main}, and then state the local porosity of the free boundary. Without special statements,  $u$ always denotes the weak solution of   equation \eqref{main}.

	The first result is concerned with the optima growth of solutions near the free boundary.
	\begin{proposition}\label{grow}
		Let $y \in \Upsilon^+$ and $B_{r_1}(y) \subset \subset \Omega$  for some $r_1 > 0$. Then there exists a positive constant $C_1$ depending only on $N$, $p$, $\lambda$, $a_0$, and $a_1$ such that
		\begin{align*}
			|u(x)| \leq C_1 |x - y|^{\frac{p}{p-1}},\forall x \in B_r(y) \quad\text{and}\quad |\nabla u(x)| \leq C_1 |x - y|^{\frac{1}{p-1}},  { \forall x \in B_r(y) }
		\end{align*}
		hold  for every $r \in (0, r_1)$.
	\end{proposition}
	
	\begin{proof}
		Since $u$ is continuous and bouned and $\lambda>1$, the main nonlinearity of the inhomogeneous  term near the free boundary is determined  by   $m_1\chi_{u>0}$; see, e.g., \cite{Hu:2025}. Thus, the   blow-up technique can be applied to establish the optimal growth  as in \cite{Karp:2000,Challal:2009,Hu:2025,Zheng:2013}. Since the proof is standard,   the details are  omitted here.
	\end{proof}
	
	The following result gives the non-degeneracy of weak solutions of  equation \eqref{main} near the free boundary $\Upsilon^+$.
	
	\begin{proposition}\label{nondegeneracy}
		Let $y \in \Upsilon^+$ and let $B_{r_2} \subset\subset \Omega$ with $	r_2 \leq \min\left\{r_1, \frac{1}{a_1} , \left(\frac{m_1}{2m_2 C_1^{\lambda-1}}\right)^{\frac{p-1}{p(\lambda-1)}}\right\}$. Then there exists a positive constant $C_2$ depending only on $N$, $p$, $m_1$, $a_1$, and $r_2$ such that
		\begin{align}
			\sup_{\partial B_r(y) \cap \{u > 0\}} u \geq C_2 r^{\frac{p}{p-1}},\forall r \in \left( 0, \frac{r_2}{2}\right) ,\notag
		\end{align}
		where $r_1$ and $C_1$ are positive constants the same as in Proposition~\ref{grow}.
	\end{proposition}
	
	\begin{proof}
		Let $y\in\{u>0\}$. For any $x_0\in\Upsilon^+$ satisfying $|y-x_{0}|\leq\frac{r_2}{2}$, we have $B_{r}(y)\subset B_{r_2}(x_{0})$ with $r<\frac{r_2}{2}$. Define $C_2:=\frac{p-1}{p}\left(\frac{m_1}{2N(a_1+1)}\right)^{\frac{1}{p-1}}$ and $v(x):=C_2|x-y|^{\frac{p}{p-1}}$ in $B_{r}(y)$. By  direct computations, we have
		\begin{align}
			\nabla v&=C_2\frac{p}{p-1}|x-y|^{\frac{2-p}{p-1}}(x-y),\notag\\
			|\nabla v|&=C_2\frac{p}{p-1}|x-y|^{\frac{1}{p-1}},\notag\\
			|\nabla v|^{p-2}&=C_2^{p-2}\left( \frac{p}{p-1}\right) ^{p-2}|x-y|^{\frac{p-2}{p-1}}\notag.
		\end{align}
		
		Noting that $r_2 < \frac{1}{a_1}$ and using the definition of $C_2$,  we obtain
		\begin{align}
			\mathrm{div}\left(a(x)|\nabla v|^{p-2}\nabla v\right)
			=
			&\mathrm{div}\left(C_2^{p-1}\left(\frac{p}{p-1}\right)^{p-1}a(x)(x-y)\right)
			\notag\\ =
			&C_2^{p-1}\left(\frac{p}{p-1}\right)^{p-1}\mathrm{div}\left( a(x)(x-y)\right)
			\notag\\ =
			&NC_2^{p-1}\left(\frac{p}{p-1}\right)^{p-1}a(x) + C_2^{p-1}\left(\frac{p}{p-1}\right)^{p-1}(x-y)\nabla a(x)
			\notag\\\leq
			& Na_1C_2^{p-1}\left(\frac{p}{p-1}\right)^{p-1} + Na_1r_2C_2^{p-1}\left(\frac{p}{p-1}\right)^{p-1}
			\notag\\ \leq
			&Na_1C_2^{p-1}\left(\frac{p}{p-1}\right)^{p-1} + NC_2^{p-1}\left(\frac{p}{p-1}\right)^{p-1}
			\notag\\ =
			&N(a_1+1)C_2^{p-1}\left(\frac{p}{p-1}\right)^{p-1}
			\notag\\ =
			&\frac{m_1}{2}\   \text{in}\   B_{r}(y).\label{6.1}
		\end{align}
		Using the fact that  $B_r(y) \subset B_{r_2}(x_0)$ and Proposition~\ref{grow}, we deduce that
		\begin{align*}
			\mathrm{div}\left(a(x)|\nabla u|^{p-2}\nabla u\right)
			=
			&  m_1 - m_2 u^{\lambda-1}
			\notag\\ \geq
			&  m_1 -  m_2 C_1^{\lambda-1} r_{2}^{\frac{p(\lambda-1)}{p-1}}
			\notag\\ \geq
			&\frac{m_1}{2}\   \text{in}\   B_r(y) \cap \{u > 0\},
		\end{align*}
		which, along with \eqref{6.1}, implies that
		\begin{align}
			-\mathrm{div}\left(a(x)|\nabla u|^{p-2}\nabla u\right) \leq -\mathrm{div}\left(a(x)|\nabla v|^{p-2}\nabla v\right)\   \text{in}\   B_r(y) \cap \{u > 0\}.\notag
		\end{align}
		Note that the definitions of $u$ and $v$ ensure that
		\begin{align}
			u=0\leq v \text{ on } B_{r}(y)\cap\Upsilon^{+}.\notag
		\end{align}
		If $u\leq v$ on $\partial B_{r}(y)\cap\{u>0\}$, then by the comparison principle, we obtain
		\begin{align}
			u\leq v  \text{ in } B_{r}(y)\cap\{u>0\}.\notag
		\end{align}
		However, $u(y)>0=v(y)$, which yields a contradiction. Therefore, there exists $y_{0}\in\partial B_{r}(y)\cap\{u>0\}$ such that
		\begin{align}
			u(y_{0})\geq v(y_{0})=C_2 r^{\frac{p}{p-1}},\notag
		\end{align}
		so we have
		\begin{align}\label{6.7}
			\sup_{\partial B_{r}(y)\cap\{u>0\}}u\geq C_2r^{\frac{p}{p-1}}.
		\end{align}
		Now for $y \in \partial\{u > 0\}$, we take a sequence $\{y_i\} \subset \{u > 0\}$ such that $y_i \to y$ as $i \to \infty$. By \eqref{6.7}, we have
		\begin{align}
			\sup_{\partial B_r(y_i) \cap \{u > 0\}} u \geq C_2 r^{\frac{p}{p-1}},\notag
		\end{align}
		which, along with the continuity of $u$, implies the desired result.
	\end{proof}
	
	Recall that a set $ P \subset \mathbb{R}^N $ is called porous with porosity constant $\delta$ if there exist constants $ r_0 > 0 $ and $ \delta > 0 $ such that
	\begin{align}
		\forall x \in P,  \forall r \in (0, r_0) \quad \Rightarrow \quad
		\exists y \in \mathbb{R}^N \text{ such   that } B_{\delta r}(y) \subset B_r(x) \setminus P.\notag
	\end{align}
	It is known that the Hausdorff dimension of a porous set  is no more than $N - c\delta^N$ with $c$ depending only $N$; {see, e.g., \cite{Koskela:1997,Martio:1987}.}

	For the obstacle problem \eqref{main},
	the optimal growth and non-degeneracy of solutions imply the local porosity of the free boundary, which is stated as the following theorem.

	\begin{theorem}\label{duo}
		Let $ u $ be  a  solution of equation \eqref{main}. Then, for every compact set $ K \subset \Omega $, the intersection $ K \cap \partial\{u > 0\} $ is porous with porosity constant $ \delta $ depending only on $N$, $p$, $\lambda$, $m_1$, $m_2$, $a_0$, $a_1$,  and $ \Omega $.
	\end{theorem}
	\begin{proof}
		The proof is standard; see, e.g.,  \cite{Karp:2000,Challal:2009,Hu:2024}.
	\end{proof}

	\section{The Corresponding Penalized Problem}\label{Sec.6}

	In section,  we study properties  of  solutions $u_\varepsilon$ to the   penalty problem  \eqref{main1}, corresponding which  the limit of  $u_\varepsilon$  in $W^{1,p}_0(\Omega)$ is  a solution, denoted by $u$, to the obstacle problem \eqref{main} (see the proof of Proposition~\ref{exi}). Note that $u\geq 0$ a.e. in $\Omega$, thus $u_\varepsilon \geq 0$ a.e. in $\Omega$. Moreover, proceeding in the same way as in the proof  of  Proposition~\ref{exi-2}, it can   be shown that  and
	$u_\varepsilon$ is uniformly bounded with
	$\|u_\varepsilon\|_{L^\infty(\Omega)}$   depending only on $N$, $p$, $\lambda$, $m_1$, $m_2$, $a_0$, and  $\Omega$. Furthermore,  a similar proof of  Theorem \ref{C^1} ensures that  $u_\varepsilon\in C^{1,\alpha}_{\mathrm{loc}}(\Omega)$ with some $\alpha \in (0,1)$ and that for any $\Omega' \subset \subset \Omega$,   $\|u_\varepsilon\|_{C^{1,\alpha}(\Omega')}$  {depends only on } $N$, $p$, $\lambda$, $m_1$, $m_2$, $a_0$, $\Omega'$, and $\Omega$. In particular, there exists a positive constant $c_1$  independent of $\varepsilon$  such that
	\begin{align}\label{c1}
		\|u_\varepsilon\|_{L^\infty(\Omega')} + \|\nabla u_\varepsilon\|_{L^\infty(\Omega')} \leq c_1.
	\end{align}
	In addition,    for any $\beta \in (0,\alpha)$ and any relatively compact  set $\Omega'' \subset \subset \Omega'$, $u_\varepsilon \to u$ in $C^{1,\beta}(\Omega'')$.  In the sequel, we always denote by $u_\varepsilon$ a non-negative   weak solution of equation \eqref{main1} and satisfies \eqref{c1}.

	The following result indicates that every weak solution   of the penalty problem  \eqref{main1} is actually a strong solution.
	\begin{proposition}\label{a.e.}
		The weak function $u_\varepsilon$ satisfies \eqref{main1} almost everywhere in $\Omega$.
	\end{proposition}

	\begin{proof}
		By Lemma~\ref{lem5}, it suffices to show that
		\begin{align}
			\left( a(x) |\nabla u_\varepsilon|^{p-2} \nabla u_\varepsilon \right) _{x_i}\in \left( L^1_{\mathrm{loc}}(\Omega)\right) ^N.\notag
		\end{align}
		Indeed, a direct computation yields
		\begin{align}
			\left| \left( a(x) |\nabla u_\varepsilon|^{p-2} \nabla u_\varepsilon \right)_{x_i} \right|
			=
			& \left| a_{x_i}(x) |\nabla u_\varepsilon|^{p-2} \nabla u_\varepsilon +(p-1) a(x)  |\nabla u_\varepsilon|^{p-2} \nabla u_{\varepsilon x_i}\right|
			\notag\\ \leq
			& |a_{x_i}(x)| |\nabla u_\varepsilon|^{p-1}  + (p-1) a(x) |\nabla u_\varepsilon|^{p-2} |D^2 u_\varepsilon|.\label{5.3}
		\end{align}
		
		Define $\Psi:=u_{\varepsilon x_{i}}\phi^{2}$ with $\phi\in C^{\infty}_{0}(B_{3r/5})$ satisfying
		\begin{align}
			\begin{cases}
				0\leq\phi\leq 1  \text{ in } B_{3r/5}, \\
				\phi=1           \text{ in } B_{r/2}, \\
				|\nabla\phi|\leq\frac{4}{r}  \text{ in } B_{3r/5},
			\end{cases}\notag
		\end{align}
		where $B_{r}\subset \subset\Omega$ is a ball with some fixed radius $r>0$.
		
		Differentiating \eqref{main1} with respect to $x_i$, then multiplying the result by $\Psi$ and integrating over $B_{3r/5}$, we obtain
		\begin{align}
			I_1 := &\int_{B_{3r/5}} \left( a(x) |\nabla u_\varepsilon|^{p-2} \nabla u_\varepsilon \right)_{x_i}  \nabla \Psi  \mathrm{d}x
			\notag \\ =
			& \int_{B_{3r/5}} \left( -m_1 \chi_\varepsilon(u_\varepsilon) + m_2 (u_\varepsilon^+)^{\lambda-1} \right)_{x_i} \Psi  \mathrm{d}x
			\notag \\ =
			& \int_{B_{3r/5}} m_2\left( (u_\varepsilon^+)^{\lambda-1} \right)_{x_i} \Psi  \mathrm{d}x -\int_{B_{3r/5}} m_1 \left( \chi_\varepsilon(u_\varepsilon) \right)_{x_i} \Psi  \mathrm{d}x
			\notag \\ =
			&: I_2 + I_3. \label{5.5}
		\end{align}
		
		Using the definition of $\Psi$, the term $I_1$ expands as
		\begin{align}
			I_1
			=
			&\int_{B_{3r/5}}(a(x)|\nabla u_\varepsilon|^{p-2}\nabla u_\varepsilon)_{x_i} \nabla \Psi \mathrm{d}x
			\notag \\ =
			&\int_{B_{3r/5}}\left((p-1) a(x)  |\nabla u_\varepsilon|^{p-2} \nabla u_{\varepsilon x_i}+ a_{x_i}(x) |\nabla u_\varepsilon|^{p-2} \nabla u_\varepsilon \right)\left(\nabla u_{\varepsilon x_i}\phi^2 +2 u_{\varepsilon x_i}\phi\nabla \phi \right) \mathrm{d}x
			\notag \\ =
			&\int_{B_{3r/5}} (p-1) a(x)  |\nabla u_\varepsilon|^{p-2} |\nabla u_{\varepsilon x_i}|^2\phi^2\mathrm{d}x +\int_{B_{3r/5}}2(p-1) a(x)  |\nabla u_\varepsilon|^{p-2} \nabla u_{\varepsilon x_i} u_{\varepsilon x_i} \phi\nabla \phi\mathrm{d}x
			\notag \\
			&+\int_{B_{3r/5}} a_{x_i}(x) |\nabla u_\varepsilon|^{p-2} \nabla u_\varepsilon \nabla u_{\varepsilon x_i}\phi^2\mathrm{d}x + \int_{B_{3r/5}} 2a_{x_i}(x) |\nabla u_\varepsilon|^{p-2} \nabla u_{\varepsilon} u_{\varepsilon x_i}\phi\nabla \phi \mathrm{d}x
			\notag\\ =
			:&I_{11}+I_{12}+I_{13}+I_{14}. \label{5.6}
		\end{align}
		
		Since $p \geq 2$, it follows that
		\begin{align}
			I_{11}
			=
			&\int_{B_{3r/5}} (p-1) a(x)  |\nabla u_\varepsilon|^{p-2} |\nabla u_{\varepsilon x_i}|^2\phi^2\mathrm{d}x
			\notag \\ \geq
			&\int_{B_{3r/5}}  a_0  |\nabla u_\varepsilon|^{p-2} |\nabla u_{\varepsilon x_i}|^2\phi^2\mathrm{d}x. \label{5.7}
		\end{align}
		
		Combining the definition of $\Psi$, Young's inequality, and \eqref{c1}, we obtain
		\begin{align}
			|I_{12}|
			\leq
			&\int_{B_{3r/5}}2(p-1) a_1  |\nabla u_\varepsilon|^{p-2} |\nabla u_{\varepsilon x_i}| |u_{\varepsilon x_i}| |\phi||\nabla \phi|\mathrm{d}x
			\notag \\ \leq
			& \frac{8}{r}\int_{B_{3r/5}}(p-1) a_1  |\nabla u_\varepsilon|^{p-2} |\nabla u_{\varepsilon x_i}| |u_{\varepsilon x_i}| |\phi|\mathrm{d}x
			\notag \\ \leq
			& \frac{a_0}{2}\int_{B_{3r/5}}  |\nabla u_\varepsilon|^{p-2} |\nabla u_{\varepsilon x_i}|^2  |\phi|^2\mathrm{d}x + \frac{32(p-1)^2 a_1^2}{a_0 r^2}\int_{B_{3r/5}}  |\nabla u_\varepsilon|^{p-2} | u_{\varepsilon x_i}|^2 \mathrm{d}x
			\notag \\ \leq
			& \frac{a_0}{2}\int_{B_{3r/5}}  |\nabla u_\varepsilon|^{p-2} |\nabla u_{\varepsilon x_i}|^2  |\phi|^2\mathrm{d}x + \frac{32(p-1)^2 a_1^2}{a_0 r^2}\int_{B_{3r/5}}  |\nabla u_\varepsilon|^p \mathrm{d}x
			\notag \\ \leq
			& \frac{a_0}{2}\int_{B_{3r/5}}  |\nabla u_\varepsilon|^{p-2} |\nabla u_{\varepsilon x_i}|^2  |\phi|^2\mathrm{d}x + Cr^{N-2},\label{5.8}
		\end{align}
		where $C$ is a positive  constant depending only on $p$, $a_0$, $a_1$, and $c_1$.

		By Young's inequality with $\varepsilon_1>0$, and \eqref{c1}, we have
		\begin{align}
			|I_{13}|
			=
			&\int_{B_{3r/5}} |a_{x_i}(x)| |\nabla u_\varepsilon|^{p-1} | \nabla u_{\varepsilon x_i}| |\phi|^2\mathrm{d}x
			\notag\\ \leq
			& \int_{B_{3r/5}} \frac{\varepsilon_1}{2} |\nabla u_\varepsilon|^{p-2} | \nabla u_{\varepsilon x_i}|^2 |\phi|^2\mathrm{d}x +  \int_{B_{3r/5}}\frac{1}{2\varepsilon_1}| a_{x_i}(x)|^2 |\nabla u_\varepsilon|^p |  |\phi|^2\mathrm{d}x
			\notag\\ \leq
			& \frac{\varepsilon_1}{2}\int_{B_{3r/5}}|\nabla u_\varepsilon|^{p-2} | \nabla u_{\varepsilon x_i}|^2 |\phi|^2\mathrm{d}x +  \frac{a_1^2}{2\varepsilon_1}\int_{B_{3r/5}} |\nabla u_\varepsilon|^p |  |\phi|^2\mathrm{d}x.	\label{5.9}
		\end{align}
		
		Choosing a suitable $\varepsilon_1$, equation \eqref{5.9} becomes
		\begin{align}\label{5.10}
			|I_{13}| \leq \frac{a_0}{4}\int_{B_{3r/5}}|\nabla u_\varepsilon|^{p-2} | \nabla u_{\varepsilon x_i}|^2 |\phi|^2\mathrm{d}x + Cr^N ,	
		\end{align}
		where $C$ is a positive  constant depending  only on  $a_0$, $a_1$, and $c_1$.
		
		Using the definition of $\phi$ and \eqref{c1}, we obtain
		\begin{align}
			|I_{14}|
			=
			&\int_{B_{3r/5}} 2a_{x_i}(x) |\nabla u_\varepsilon|^{p-2} \nabla u_{\varepsilon} u_{\varepsilon x_i}\phi\nabla \phi \mathrm{d}x
			\notag\\ \leq
			& \int_{B_{3r/5}} \frac{8a_1}{r} |\nabla u_\varepsilon|^p \mathrm{d}x
			\notag\\\leq
			& Cr^{N-1},\label{5.11}
		\end{align}
		{where $C$ is a positive  constant depending  only on  $a_1$ and $c_1$.}	
		
		From the definition of $\phi$, \eqref{c1}, and the Poincar\'e's inequality, we deduce that
		\begin{align}
			I_2
			=
			&\int_{B_{3r/5}} m_2\left( (u_\varepsilon^+)^{\lambda-1} \right)_{x_i} \Psi  \mathrm{d}x
			\notag\\ \leq
			& \int_{B_{3r/5}} m_2(\lambda-1) (u_\varepsilon)^{\lambda-2} (u_{\varepsilon x_i})^2 \phi^2  \mathrm{d}x
			\notag\\ \leq
			& \int_{B_{3r/5}} m_2(\lambda-1)c_1^\lambda  \mathrm{d}x
			\notag\\ \leq
			&Cr^N,\label{5.12}
		\end{align} 	
		where $C$ is a positive  constant depending  only on $N$, $\lambda$,  $m_2$, $c_1$,  and $r$.
		
		Note that $\chi'_\varepsilon(s) \geq 0$. It follows that
		\begin{align}
			I_3
			=
			&-\int_{B_{3r/5}} m_1 \left( \chi_\varepsilon(u_\varepsilon) \right)_{x_i} \Psi  \mathrm{d}x
			\notag\\=
			&-\int_{B_{3r/5}} m_1 \chi'_{\varepsilon} (u_\varepsilon) (u_{\varepsilon x_i})^2 \phi^2\mathrm{d}x
			\notag\\ \leq
			& 0 .\label{5.13}
		\end{align}
		
		Using \eqref{5.5}, \eqref{5.12}, and \eqref{5.13}, we have
		\begin{align}\label{5.14}
			I_1=I_2+I_3\leq Cr^N,
		\end{align}
		where $C$ is a positive  constant depending  only on $N$, $\lambda$, $m_2$, $c_1$,  and $r$.
		
		From \eqref{5.6}, \eqref{5.7}, \eqref{5.8}, \eqref{5.10}, and \eqref{5.11}, we infer that
		\begin{align}
			I_1
			\geq
			& I_{11}-|I_{12}|-|I_{13}|-|I_{14}|
			\notag\\\geq
			& \int_{B_{3r/5}} \frac{ a_0 }{4} |\nabla u_\varepsilon|^{p-2} |\nabla u_{\varepsilon x_i}|^2\phi^2\mathrm{d}x  - Cr^{N-2}
			\notag\\ \geq
			& \int_{B_{r/2}} \frac{ a_0 }{4} |\nabla u_\varepsilon|^{p-2} |\nabla u_{\varepsilon x_i}|^2\phi^2\mathrm{d}x  - Cr^{N-2}
			\notag\\ =
			& \int_{B_{r/2}} \frac{ a_0 }{4} |\nabla u_\varepsilon|^{p-2} |\nabla u_{\varepsilon x_i}|^2\mathrm{d}x  - Cr^{N-2},\label{5.15}
		\end{align}
		where $C$ is a positive  constant depending  only on $p$, $a_0$, $a_1$, and $c_1$.
		
		{By \eqref{5.14} and \eqref{5.15},}	 we obtain
		\begin{align}\label{5.16}
			\int_{B_{r/2}}  |\nabla u_\varepsilon|^{p-2} |\nabla u_{\varepsilon x_i}|^2\mathrm{d}x  \leq Cr^{N-2},
		\end{align}
		where $C$ is a positive  constant depending  only on $N$, $p$, $\lambda$, $m_2$, $c_1$,   $a_0$,  $a_1$,  and $r$.
		
		Summing \eqref{5.16} over $i=1$ to $N$ leads to
		\begin{align}\label{5.17}
			\int_{B_{r/2}}  |\nabla u_\varepsilon|^{p-2} |D^2u_\varepsilon |\mathrm{d}x  \leq \int_{B_{r/2}} \frac{1}{2} |\nabla u_\varepsilon|^{p-2} |D^2u_\varepsilon |^2\mathrm{d}x + \int_{B_{r/2}} \frac{1}{2} |\nabla u_\varepsilon|^{p-2} \mathrm{d}x\leq Cr^{N-2}.
		\end{align}
		Note that \eqref{c1} ensures that
		\begin{align}\label{5.18}
			\int_{B_{r/2}}|a_{x_i}(x)| |\nabla u_\varepsilon|^{p-1} \leq a_1 \int_{B_{r/2}} |\nabla u_\varepsilon|^{p-1} \leq Cr^N,
		\end{align}
		where $C$ is a positive  constant depending  only on  $p$, $a_1$, and $c_1$.
		
		{	From \eqref{5.17}, \eqref{5.18}, and the boundedness of $a $, } we deduce that
		\begin{align}
			|a_{x_i}(x)| |\nabla u_\varepsilon|^{p-1} + (p-1) a(x) |\nabla u_\varepsilon|^{p-2} |D^2 u_\varepsilon| \in L^1_{\mathrm{loc}}(\Omega),\notag
		\end{align}
		which, along with \eqref{5.3}, implies that
		\begin{align}
			\left( a(x) |\nabla u_\varepsilon|^{p-2} \nabla u_\varepsilon \right)_{x_i} \in \left( L^1_{\mathrm{loc}}(\Omega)\right) ^N.\notag
		\end{align}
	\end{proof}
	
	The following proposition presents a quantitative lower bound that characterizes the non-degeneracy of both the gradient and the Hessian of the penalized solution.
	\begin{proposition}\label{ho1}
		It holds that
		\begin{align}
			\frac{1}{C_3}\left| m_1 \chi_\varepsilon \left( u_\varepsilon \right) - m_2 \left(  u_{\varepsilon}^{+} \right) ^{\lambda-1}\right| ^2
			\leq |\nabla u_\varepsilon|^{2(p-1)}+\left(|\nabla u_\varepsilon|^{p-2} |D^2 u_\varepsilon|\right) ^2\ \text{a.e.\ in}\ \Omega,\notag
		\end{align}
		where  $C_3$  is a positive constant depending only on $N$, $p$,  and $a_1$.
	\end{proposition}
	
	\begin{proof}By direct calculation, we obtain
		\begin{align}
			&\left| m_1 \chi_\varepsilon\left( u_\varepsilon \right) - m_2 \left( u_{\varepsilon}^{+} \right)^{\lambda-1} \right|^2
			\notag\\ =
			& \left( \mathrm{div}\left(  a(x)|\nabla u_\varepsilon|^{p-2} \nabla u_\varepsilon \right)  \right)^2
			\notag\\=
			&\left( \sum^N_{i=1}\left( \frac{\partial \left( a(x) |\nabla u_\varepsilon|^{p-2}\right) }{\partial x_i} u_{\varepsilon x_i} + a(x)  |\nabla u_\varepsilon|^{p-2} u_{\varepsilon x_i x_i}\right) \right) ^2
			\notag\\=
			&\left( \sum^N_{i=1}\left( a_{x_i}(x) |\nabla u_\varepsilon|^{p-2}u_{\varepsilon x_i} + (p-2)a(x)|\nabla u_\varepsilon|^{p-4}\left( \sum^N_{j=1}u_{\varepsilon x_j}u_{\varepsilon x_j x_i}\right)u_{\varepsilon x_i} + a(x)  |\nabla u_\varepsilon|^{p-2} u_{\varepsilon x_i x_i}\right)\right)^2
			\notag\\ \leq
			&\left( \sum^N_{i=1}\left( a_1 |\nabla u_\varepsilon|^{p-1} + (p-2)a_1|\nabla u_\varepsilon|^{p-2}|\nabla u_{\varepsilon x_i}| + a_1  |\nabla u_\varepsilon|^{p-2} |u_{\varepsilon x_i x_i}|\right)\right)^2
			\notag\\ \leq
			&\left( N a_1 |\nabla u_\varepsilon|^{p-1} + \sum^N_{i=1}\left((p-2)a_1|\nabla u_\varepsilon|^{p-2}|\nabla u_{\varepsilon x_i}| + a_1  |\nabla u_\varepsilon|^{p-2}| u_{\varepsilon x_i x_i}|\right)\right)^2
			\notag\\ \leq
			&\left( N a_1 |\nabla u_\varepsilon|^{p-1} + Na_1(p-1)|\nabla u_\varepsilon|^{p-2} |D^2 u_\varepsilon| \right)^2 \label{1111}
			\\ \leq
			& C_3\left(  |\nabla u_\varepsilon|^{2(p-1)}  + \left( |\nabla u_\varepsilon|^{p-2} |D^2 u_\varepsilon| \right)^2\right) . \notag
		\end{align}
		This completes the proof.
	\end{proof}

	\section{Hausdorff measure of the free boundary}\label{Sec.7}

	In this section,  we prove that   at least one solution $u$  to the obstacle problem \eqref{main} corresponds  the free boundary having a locally finite $(N-1)$-dimensional Hausdorff measure.
	More specifically, we consider  $u_{\varepsilon} \to u$ in $C^{1,\alpha}_{\mathrm{loc}}(\Omega)$ with $u_{\varepsilon}$ and $u$  being {the solutions of the penalty problem \eqref{main1} and the obstacle problem \eqref{main},} respectively.
	First, by virtue of Proposition~\ref{grow}, for $\sigma\in(0,1)$, we define
	\begin{align}
		O_{\sigma}:=\left\{|\nabla u|\le\sigma^{\frac{1}{p-1}}\right\}, \quad
		O_{\sigma i}:=\left\{|u_{x_i}|\le\sigma^{\frac{1}{p-1}}\right\}.\notag
	\end{align}

	The following proposition provides a crucial measure-theoretic estimate, which   serves as a foundation for investigating the Hausdorff measure of the free boundary.
	\begin{proposition}\label{ho2}
		Let $C_1$ and $r_2$ be the positive constants determined by Proposition~\ref{grow} and Proposition~\ref{nondegeneracy}, respectively. Let $B_{r_4}\subset\subset\Omega$ be a ball whose radius $r_4$ is sufficiently small and satisfies $r_4\leq r_2$.	Then   for every
		$x_{0}\in\Upsilon^+ \cap B_{r_4}$, any $\sigma\in(0,1)$, and any $r\in(0,r_4)$ with $B_{r}(x_{0})\subset B_{r_4}$, it holds that
		\begin{align}
			\int_0^1\mathcal{L}^N \left( O_{\sigma}\cap B_{rs}(x_0)\cap\{u>0\}\right)\mathrm{d}s\leq C\sigma r^{N-1},\notag
		\end{align}
		where $C$ is a positive constant depending only on  $N$, $p$,  $\lambda$, $m_1$, $m_2$, $a_0$,  $a_1$, $c_1$,   and $\Omega$.
	\end{proposition}
	
	\begin{proof}Define
		\begin{align}
			O_{\varepsilon}:=\left\{ |\nabla u_\varepsilon| \leq 2\sigma^{\frac{1}{p-1}}\right\},
			\quad
			O_{\varepsilon i}:=\left\{ |u_{\varepsilon x_i}| \leq 2\sigma^{\frac{1}{p-1}}\right\}.\notag
		\end{align}
		
		We claim that $(O_\sigma\cap B_{r_4})\subset(O_\varepsilon \cap B_{r_4})$.
		Indeed, there exists $\varepsilon_1>0$ such that for every $\varepsilon\in(0,\varepsilon_1)$,
		\begin{align}
			\left\| |\nabla u_\varepsilon| - |\nabla u|\right\| _{L^{\infty}(\overline{B}_{r_4})}\leq \sigma^{\frac{1}{p-1}}.\notag
		\end{align}
		Then for any $x\in O_{\sigma}\cap B_{r_4}$, we have
		\begin{align}
			\left\| \nabla u_\varepsilon\right\| _{L^{\infty}(\overline{B}_{r_4})}\leq 	\left\| |\nabla u_\varepsilon| - |\nabla u|\right\| _{L^{\infty}(\overline{B}_{r_4})}\ + 	\left\| \nabla u \right\| _{L^{\infty}(\overline{B}_{r_4})}\leq 2\sigma^{\frac{1}{p-1}},\notag
		\end{align}
		which implies that $x \in O_\varepsilon \cap B_{r_4}$, or equivalently, $(O_\sigma \cap B_{r_4}) \subset (O_\varepsilon \cap B_{r_4})$.
		
		Let $H$ be defined by
		\begin{align}
			H(\eta) :=
			\begin{cases}
				2^{p-1}\sigma, & \eta > 2\sigma^{\frac{1}{p-1}}, \\
				2^{p-2}\sigma^{\frac{p-2}{p-1}}\eta, & |\eta|\leq 2\sigma^{\frac{1}{p-1}}, \\
				-2^{p-1}\sigma, & \eta < -2\sigma^{\frac{1}{p-1}}.
			\end{cases}\notag
		\end{align}
		It is clear that
		\begin{align}\label{7.1}
			|H(\eta)|\leq2^{p-1}\sigma.
		\end{align}

		Differentiating  \eqref{main1} with respect to $x_i$, we obtain
		\begin{align}\label{7.2}
			- \mathrm{div}\left(a_{x_i}(x) |\nabla u_\varepsilon|^{p-2} \nabla u_\varepsilon +(p-1) a(x)  |\nabla u_\varepsilon|^{p-2} \nabla u_{\varepsilon x_i}\right) +m_1 \chi_\varepsilon' \left( u_\varepsilon \right)u_{\varepsilon x_i} - m_2 (\lambda-1) \left(  u_{\varepsilon}^{+} \right) ^{\lambda-2}u_{\varepsilon x_i}=0.
		\end{align}
		Multiplying \eqref{7.2} by $H(u_{\varepsilon x_i})$ and integrating over $B_{rs}(x)$, we obtain
		\begin{align}
			&\int_{B_{rs}(x_0)}\left(a_{x_i}(x) |\nabla u_\varepsilon|^{p-2} \nabla u_\varepsilon +(p-1) a(x)  |\nabla u_\varepsilon|^{p-2} \nabla u_{\varepsilon x_i}\right)\nabla H(u_{\varepsilon x_i}) \mathrm{d}x
			\notag\\
			&+ \int_{B_{rs}(x_0)} m_1 \chi_\varepsilon' \left( u_\varepsilon \right)u_{\varepsilon x_i}H(u_{\varepsilon x_i})\mathrm{d}x - \int_{B_{rs}(x_0)}m_2 (\lambda-1) \left(  u_{\varepsilon}^{+} \right) ^{\lambda-2}u_{\varepsilon x_i}H(u_{\varepsilon x_i})\mathrm{d}x
			\notag\\=
			&\int_{\partial B_{ rs}(x_0)}\left(a_{x_i}(x) |\nabla u_\varepsilon|^{p-2} \nabla u_\varepsilon +(p-1) a(x)  |\nabla u_\varepsilon|^{p-2} \nabla u_{\varepsilon x_i}\right) H(u_{\varepsilon x_i})  \nu    \mathrm{d}S ,\notag
		\end{align}
		where $\nu$ is the unit outward normal vector.
		
		Let
		\begin{align}
			&I_1:=\int_{B_{rs}(x_0)}\left(a_{x_i}(x) |\nabla u_\varepsilon|^{p-2} \nabla u_\varepsilon +(p-1) a(x)  |\nabla u_\varepsilon|^{p-2} \nabla u_{\varepsilon x_i}\right)\nabla H(u_{\varepsilon x_i}) \mathrm{d}x,\notag\\
			&I_2:=\int_{B_{rs}(x_0)} m_1 \chi_\varepsilon' \left( u_\varepsilon \right)u_{\varepsilon x_i}H(u_{\varepsilon x_i})\mathrm{d}x,\notag\\
			&I_3:=\int_{B_{rs}(x_0)}m_2 (\lambda-1) \left(  u_{\varepsilon}^{+} \right) ^{\lambda-2}u_{\varepsilon x_i}H(u_{\varepsilon x_i})\mathrm{d}x,\notag\\
			&I_4:=\int_{\partial B_{ rs}(x_0)}\left(a_{x_i}(x) |\nabla u_\varepsilon|^{p-2} \nabla u_\varepsilon +(p-1) a(x)  |\nabla u_\varepsilon|^{p-2} \nabla u_{\varepsilon x_i}\right) H(u_{\varepsilon x_i})  \nu
			\mathrm{d}S .\notag
		\end{align}
		It follows that
		\begin{align}\label{7.3}
			I_1 + I_2 - I_3 = I_4.
		\end{align}
		
		{We now estimate $I_1$, $I_2$, $I_3$, and $I_4$, respectively.}	 	From the definition of $H(\eta)$, it is clear that $H(\eta)\eta \geq 0$. Consequently, for $I_2$ we obtain
		\begin{align}\label{7.4}
			I_2=\int_{B_{rs}(x_0)} m_1 \chi_\varepsilon' \left( u_\varepsilon \right)u_{\varepsilon x_i}H(u_{\varepsilon x_i})\mathrm{d}x \geq 0.
		\end{align}
		
		Combining \eqref{c1} and \eqref{7.1}, we find that $I_3$ satisfies
		\begin{align}
			I_3
			=
			&\int_{B_{rs}(x_0)}m_2 (\lambda-1) \left(  u_{\varepsilon}^{+} \right) ^{\lambda-2}u_{\varepsilon x_i}H(u_{\varepsilon x_i})\mathrm{d}x
			\notag\\ \leq
			& \int_{B_{rs}(x_0)}m_2 (\lambda-1) \left|  u_{\varepsilon}^{+} \right| ^{\lambda-2} |u_{\varepsilon x_i}||H(u_{\varepsilon x_i})|\mathrm{d}x
			\notag\\ \leq
			& m_2 (\lambda-1)c_{1} ^{\lambda-1}\int_{B_{rs}(x_0)} 2^{p-1}\sigma\mathrm{d}x
			\notag\\ \leq
			&C\sigma r^N s^N,\label{7.5}
		\end{align}
		where $C$ is a positive  constant  depending only on  $p$, $\lambda$, $m_2$,   and $c_1$.
		
		By H\"older's inequality and \eqref{5.17}, we find that $I_4$ satisfies
		\begin{align}
			\int_{0}^{1}I_4\mathrm{d}s
			=
			&\int_{0}^{1}\int_{\partial B_{ rs}(x_0)}\left(a_{x_i}(x) |\nabla u_\varepsilon|^{p-2} \nabla u_\varepsilon +(p-1) a(x)  |\nabla u_\varepsilon|^{p-2} \nabla u_{\varepsilon x_i}\right) H(u_{\varepsilon x_i})  \nu  \mathrm{d}S \mathrm{d}s
			\notag\\\leq
			& \int_{B_r(x_0)}a_1 |\nabla u_\varepsilon|^{p-1} |H(u_{\varepsilon x_i})| \mathrm{d}x + \int_{B_r(x_0)}(p-1) a_1  |\nabla u_\varepsilon|^{p-2} |D^2 u| |H(u_{\varepsilon x_i})| \mathrm{d}x
			\notag\\ \leq
			& \int_{B_r(x_0)}a_1 |\nabla u_\varepsilon|^{p-1}  2^{p-1}\sigma \mathrm{d}x + \int_{B_r(x_0)}(p-1) a_1   2^{p-1}\sigma |\nabla u_\varepsilon|^{p-2} |D^2 u|  \mathrm{d}x
			\notag\\\leq
			& C\sigma r^N + C\sigma\left(\int_{B_r(x_0)} |\nabla u_\varepsilon|^{p-2}  \mathrm{d}x  \right) ^\frac{1}{2} \left(\int_{B_r(x_0)} |\nabla u_\varepsilon|^{p-2} |D^2 u|^2 \mathrm{d}x  \right) ^\frac{1}{2}
			\notag\\ \leq
			& C\sigma r^N +  C\sigma r^{N-1}
			\notag\\ \leq
			& C\sigma r^{N-1},\label{7.6}
		\end{align}
		where $C$ is a positive  constant  depending only on $p$, $a_1$, $c_1$,    and $\Omega$.
		
		For the term $I_1$, we have
		\begin{align}
			I_1
			=
			&\int_{B_{rs}(x_0)}\left(a_{x_i}(x) |\nabla u_\varepsilon|^{p-2} \nabla u_\varepsilon +(p-1) a(x)  |\nabla u_\varepsilon|^{p-2} \nabla u_{\varepsilon x_i}\right)\nabla H(u_{\varepsilon x_i}) \mathrm{d}x
			\notag\\ =
			& \int_{B_{rs}(x_0)}a_{x_i}(x) |\nabla u_\varepsilon|^{p-2} \nabla u_\varepsilon \nabla H(u_{\varepsilon x_i}) \mathrm{d}x + \int_{B_{rs}(x_0)} (p-1) a(x)  |\nabla u_\varepsilon|^{p-2} \nabla u_{\varepsilon x_i}\nabla H(u_{\varepsilon x_i}) \mathrm{d}x. \notag
		\end{align}
		Let
		\begin{align}
			&I_{11}:=\int_{B_{rs}(x_0)}a_{x_i}(x) |\nabla u_\varepsilon|^{p-2} \nabla u_\varepsilon \nabla H(u_{\varepsilon x_i}) \mathrm{d}x,\notag\\
			&I_{12}:=\int_{B_{rs}(x_0)} (p-1) a(x)  |\nabla u_\varepsilon|^{p-2} \nabla u_{\varepsilon x_i}\nabla H(u_{\varepsilon x_i}) \mathrm{d}x. \notag
		\end{align}
		Thus,
		\begin{align}\label{7.8}
			I_1 = I_{11} + I_{12}.
		\end{align}
		By  \cite[P.711]{Evans:2022}, \eqref{1111}, and \eqref{7.6}, we find that $I_{11}$ satisfies
		\begin{align}
			\sum_{i=1}^{N} \int_0^1 I_{11}\mathrm{d}s
			=
			&\sum_{i=1}^{N} \int_0^1 \int_{B_{rs}(x_0)}a_{x_i}(x) |\nabla u_\varepsilon|^{p-2} \nabla u_\varepsilon \nabla H(u_{\varepsilon x_i}) \mathrm{d}x\mathrm{d}s
			\notag\\ \leq
			&- \sum_{i=1}^{N} \int_0^1 \int_{B_{rs}(x_0)}\mathrm{div} \left( a_{x_i}(x) |\nabla u_\varepsilon|^{p-2} \nabla u_\varepsilon \right)  H(u_{\varepsilon x_i}) \mathrm{d}x\mathrm{d}s
			\notag\\
			&+  \sum_{i=1}^{N} \int_0^1 \int_{\partial B_{ rs}(x_0)} a_{x_i}(x)  H(u_{\varepsilon x_i})|\nabla u_\varepsilon|^{p-2} \nabla u_\varepsilon  \nu  \mathrm{d}S\mathrm{d}s
			\notag\\ \leq
			& \sum_{i=1}^{N} \int_0^1 \int_{B_{rs}(x_0)}\left( N a_1 |\nabla u_\varepsilon|^{p-1} + Na_1(p-1)|\nabla u_\varepsilon|^{p-2} |D^2 u_\varepsilon| \right)  |H(u_{\varepsilon x_i})| \mathrm{d}x\mathrm{d}s
			\notag\\
			&+  \sum_{i=1}^{N}  \int_{ B_{ r}(x_0)} a_{x_i}(x)  H(u_{\varepsilon x_i})|\nabla u_\varepsilon|^{p-2} \nabla u_\varepsilon  \nu  \mathrm{d}x
			\notag\\ \leq
			& \sum_{i=1}^{N} \int_{B_{r}(x_0)}\left( N a_1 |\nabla u_\varepsilon|^{p-1} + Na_1(p-1)|\nabla u_\varepsilon|^{p-2} |D^2 u_\varepsilon| \right)  |H(u_{\varepsilon x_i})| \mathrm{d}x\mathrm{d}s
			\notag\\
			&+  \sum_{i=1}^{N}  \int_{ B_{ r}(x_0)} |a_{x_i}(x)|  |H(u_{\varepsilon x_i})||\nabla u_\varepsilon|^{p-1}     \mathrm{d}x
			\notag\\ \leq
			& C \sigma r^{N-1}  + \sum_{i=1}^{N}  \int_{ B_{ r}(x_0)}a_1  |\nabla u_\varepsilon|^{p-1} 2^{p-1}\sigma \mathrm{d}x
			\notag\\ \leq
			& C \sigma r^{N-1} , \label{7.9}
		\end{align}
		where $C$ is a positive  constant  depending only on $N$, $p$, $a_1$,  $c_1$,    and $\Omega$.
		
		Since $\nabla H(u_{\varepsilon x_i}) = 0$ in $O_{\varepsilon i}^c$ and $O_\sigma \subset O_\varepsilon \subset O_{\varepsilon i}$, the term $I_{12}$ satisfies
		\begin{align}
			\sum_{i=1}^{N} I_{12}
			=
			&\sum_{i=1}^{N} \int_{B_{rs}(x_0)} (p-1) a(x)  |\nabla u_\varepsilon|^{p-2} \nabla u_{\varepsilon x_i}\nabla H(u_{\varepsilon x_i}) \mathrm{d}x
			\notag\\ =
			&\sum_{i=1}^{N} \int_{B_{rs}(x_0)\cap O_{\varepsilon i} } (p-1) a(x)  |\nabla u_\varepsilon|^{p-2} \nabla u_{\varepsilon x_i}\nabla H(u_{\varepsilon x_i}) \mathrm{d}x
			\notag\\
			&+ \sum_{i=1}^{N} \int_{B_{rs}(x_0)\cap O_{\varepsilon i}^c } (p-1) a(x)  |\nabla u_\varepsilon|^{p-2} \nabla u_{\varepsilon x_i}\nabla H(u_{\varepsilon x_i}) \mathrm{d}x
			\notag\\=
			&\sum_{i=1}^{N} \int_{B_{rs}(x_0)\cap O_{\varepsilon i} } (p-1) a(x)  |\nabla u_\varepsilon|^{p-2} \nabla u_{\varepsilon x_i}\nabla H(u_{\varepsilon x_i}) \mathrm{d}x
			\notag\\ =
			&\sum_{i=1}^{N} \int_{B_{rs}(x_0)\cap O_{\varepsilon i} } (p-1) a(x)  |\nabla u_\varepsilon|^{p-2} |\nabla u_{\varepsilon x_i}|^2 2^{p-2}\sigma^{\frac{p-2}{p-1}} \mathrm{d}x
			\notag\\ \geq
			&\sum_{i=1}^{N} \int_{B_{rs}(x_0)\cap O_{\varepsilon } } (p-1) a(x)  |\nabla u_\varepsilon|^{p-2} |\nabla u_{\varepsilon x_i}|^2 2^{p-2}\sigma^{\frac{p-2}{p-1}} \mathrm{d}x
			\notag\\ =
			& \int_{B_{rs}(x_0)\cap O_{\varepsilon } } (p-1) a(x)  |\nabla u_\varepsilon|^{p-2} |D^2 u_{\varepsilon }|^2 2^{p-2}\sigma^{\frac{p-2}{p-1}} \mathrm{d}x
			\notag\\ \geq
			& \int_{B_{rs}(x_0)\cap O_{\varepsilon } } (p-1) a(x)  |\nabla u_\varepsilon|^{2(p-2)} |D^2 u_{\varepsilon }|^2 \mathrm{d}x
			\notag\\ \geq
			& C\int_{B_{rs}(x_0)\cap O_{\varepsilon } }   \left( |\nabla u_\varepsilon|^{(p-2)} |D^2 u_{\varepsilon }| \right) ^2\mathrm{d}x
			\notag\\ \geq
			& C\int_{B_{rs}(x_0)\cap O_{\sigma} }   \left( |\nabla u_\varepsilon|^{(p-2)} |D^2 u_{\varepsilon }| \right) ^2\mathrm{d}x ,\label{7.10}
		\end{align}
		{where $C$ is a positive  constant  depending only on  $p$ and $a_0$.}
		
		By \eqref{7.3}, \eqref{7.4}, and \eqref{7.8}, we obtain
		\begin{align}
			I_{12}=I_4 + I_3  -I_2 -I_{11}\leq I_4 + I_3 + |I_{11}|, \notag
		\end{align}
		which, along with \eqref{7.5}, \eqref{7.6}, and \eqref{7.9},  gives
		\begin{align}
			\sum_{i=1}^{N} \int_0^1 I_{12}\mathrm{d}s
			\leq
			&\sum_{i=1}^{N} \int_0^1  I_4 \mathrm{d}s + 	\sum_{i=1}^{N} \int_0^1  I_3 \mathrm{d}s + 	\sum_{i=1}^{N} \int_0^1  |I_{11}|\mathrm{d}s
			\notag\\ \leq
			& C\sigma r^{N-1} + \sum_{i=1}^{N} \int_0^1   C\sigma r^N s^N \mathrm{d}s
			\notag\\ \leq
			& C\sigma r^{N-1},\label{7.12}
		\end{align}
		where $C$ is a positive  constant  depending only on $N$, $p$, $\lambda$, $m_2$, $a_1$,  $c_1$,  and $\Omega$.
		
		Combining \eqref{7.10} and \eqref{7.12}, we obtain
		\begin{align}\label{7.13}
			\int_0^1 \int_{B_{rs}(x_0)\cap O_{\sigma} }   \left( |\nabla u_\varepsilon|^{(p-2)} |D^2 u_{\varepsilon }| \right) ^2\mathrm{d}x \mathrm{d}s \leq C\sigma r^{N-1},
		\end{align}
		where $C$ is a positive  constant  depending only on $N$, $p$, $\lambda$, $m_2$, $a_0$, $a_1$,   $c_1$,  and $\Omega$.
		
		Using Proposition~\ref{ho1}, \eqref{7.13}, the definition of $O_\varepsilon$, and the inclusion $O_\sigma \subset O_\varepsilon$, we deduce that
		\begin{align}
			&\frac{1}{C_3} \int_0^1 \int_{B_{rs}(x_0)\cap O_{\sigma} }  \left| m_1 \chi_\varepsilon \left( u_\varepsilon \right) - m_2 \left(  u_{\varepsilon}^{+} \right) ^{\lambda-1}\right| ^2 \mathrm{d}x \mathrm{d}s
			\notag\\ \leq
			&\int_0^1 \int_{B_{rs}(x_0)\cap O_{\sigma} } |\nabla u_\varepsilon|^{2(p-1)}\mathrm{d}x \mathrm{d}s
			+ \int_0^1 \int_{B_{rs}(x_0)\cap O_{\sigma} } \left(|\nabla u_\varepsilon|^{p-2} |D^2 u_\varepsilon|\right) ^2\mathrm{d}x \mathrm{d}s
			\notag\\\leq
			& C\sigma r^{N-1} + \int_0^1 \int_{B_{rs}(x_0)\cap O_{\varepsilon} } |\nabla u_\varepsilon|^{2(p-1)}\mathrm{d}x \mathrm{d}s
			\notag\\ \leq
			& C\sigma r^{N-1} + \int_0^1 \int_{B_{rs}(x_0)\cap O_{\varepsilon} } 2^{p-1}\sigma|\nabla u_\varepsilon|^{(p-1)}\mathrm{d}x \mathrm{d}s
			\notag\\ \leq
			& C\sigma r^{N-1},\notag
		\end{align}
		where $C$ is a positive  constant  depending only on $N$, $p$, $\lambda$, $m_2$, $a_0$, $a_1$,  $c_1$,  and $\Omega$.
		
		Letting $\varepsilon \to 0$, we obtain
		\begin{align}\label{7.14}
			\frac{1}{C_3} \int_0^1 \int_{B_{rs}(x_0)\cap O_{\sigma} }  \left| m_1 \chi_{\{u>0\}}  - m_2  u ^{\lambda-1}\chi_{\{u>0\}} \right| ^2 \mathrm{d}x \mathrm{d}s
			\leq  C\sigma r^{N-1},
		\end{align}
		where $C$ is a positive  constant  depending only on $N$, $p$, $\lambda$, $m_2$, $a_0$, $a_1$,   $c_1$,  and $\Omega$.

		Since $r < r_2$, by Proposition~\ref{grow}, we have
		\begin{align*}
			m_1 \chi_{\{u>0\}} - m_2 u^{\lambda-1} \chi_{\{u>0\}}
			&= m_1 - m_2 u^{\lambda-1} \\
			&\ge m_1 - m_2 C_1^{\lambda-1} r_{2}^{\frac{p(\lambda-1)}{p-1}} \\
			&\ge \frac{m_1}{2} \ \text{in } B_{r}(x_0) \cap \{u > 0\},
		\end{align*}
		which, along with \eqref{7.14}, implies that
		\begin{align}
			\int_0^1 \int_{B_{rs}(x_0)\cap O_{\sigma} \cap \{u>0\}}  \frac{m_1^2}{4C_3}\mathrm{d}x \mathrm{d}s \leq  C\sigma r^{N-1},\notag
		\end{align}
		or equivalently,
		\begin{align}
			\int_0^1\mathcal{L}^N \left( O_{\sigma}\cap B_{rs}(x_0)\cap\{u>0\}\right)\mathrm{d}s\leq C\sigma r^{N-1},\notag
		\end{align}
		where $C$ is a positive  constant  depending only on $N$, $p$, $\lambda$, $m_1$, $m_2$, $a_0$,   $a_1$,  $c_1$,  and $\Omega$.
	\end{proof}

	Now we recall the definition of Hausdorff measure of a set in $\mathbb{R}^N$.
	\begin{definition} [{\cite[p.60]{Evans:2018}}]\label{hadef}
		\begin{enumerate} \item [(i)]Let $ A \subset \mathbb{R}^N $, $ 0 \leq s < \infty $, $ 0 < \delta \leq \infty $. Define
			\begin{align}
				\mathcal{H}_\delta^t(A) := \inf \left\{ \sum_{j=1}^{\infty} \alpha(s) \left( \frac{\operatorname{diam} B_j}{2} \right)^s \ \bigg|\ A \subset \bigcup_{j=1}^{\infty} B_j,\ \operatorname{diam} B_j \leq \delta \right\},\notag
			\end{align}
			where $ \alpha(s) :=\frac{\pi^{\frac{s}{2}}}{\Gamma\left( \frac{s}{2} + 1 \right)}$ with the standard gamma function $  \Gamma (s) := \int_0^\infty e^{-x} x^{s-1} \mathrm{d}x $ for $   0<s <  \infty $.
			\item [(i)]For $ A $ and $ s $ as above, define
			\begin{align}
				\mathcal{H}^s(A) := \lim_{\delta \to 0} \mathcal{H}_\delta^s(A) = \sup_{\delta > 0} \mathcal{H}_\delta^s(A).\notag
			\end{align}
			We call that $ \mathcal{H}^s (A)$ is the $ s $-dimensional Hausdorff measure of $A$ in $ \mathbb{R}^N $.
		\end{enumerate}
	\end{definition}
	
	Based on the Proposition \ref{ho2}, we show that the free boundary $\Upsilon^+$ possesses locally finite $(N-1)$-dimensional Hausdorff measure.
	\begin{theorem} \label{ho}
		Let  $r_4$  be the same as in  Proposition~\ref{ho2}. Then  for every  $x_0 \in\Upsilon^+ \cap B_{r_4}$ {and $r \in \left( 0, \frac{r_4}{2}\right) $,} there holds
		\begin{align}
			\mathcal{H}^{N-1}\left( B_r(x_0) \cap \Upsilon^+\right)  \leq C r^{N-1},\notag
		\end{align}
		where $C$ is a positive constant depending only on  $N$, $p$, $\lambda$, $m_1$, $m_2$, $a_0$, $a_1$,   $c_1$,  and $\Omega$.
	\end{theorem}
	
	\begin{proof}
		We first claim
		\begin{align}\label{8.1}
			\mathcal{L}^N \left(  O_{\sigma}\cap B_{rs}(x_0)\cap\{u>0\}\right)  \leq C\sigma r^{N-1}, { \forall r < \frac{r_4}{2},}
		\end{align}
		where $C$ is a positive  constant  depending only on $N$, $p$, $\lambda$, $m_1$, $m_2$, $a_0$,   $a_1$,  $c_1$,  and $\Omega$. Indeed, if \eqref{8.1} fails, then there exists a ball $B_r(x_0)$ centered on the free boundary such that for every $k \in \mathbb{R}$, it holds that
		\begin{align}
			\mathcal{L}^N\left( O_\sigma \cap B_r(x_0) \cap \{ u > 0 \}\right)  \geq k \sigma r^{N-1}.\notag
		\end{align}
		However, by Proposition~\ref{ho2}, we obtain
		\begin{align}
			C \sigma r^{N-1}
			\geq
			& \int_0^1 \mathcal{L}^N\left( O_\sigma \cap B_{2rs}(x_0) \cap \{ u > 0 \}\right)  \mathrm{d}s
			\notag\\ =
			& \int_0^{\frac{1}{2}} \mathcal{L}^N\left( O_\sigma \cap B_{2rs}(x_0) \cap \{ u > 0 \}\right)  \mathrm{d}s
			+ \int_{\frac{1}{2}}^1 \mathcal{L}^N\left( O_\sigma \cap B_{2rs}(x_0) \cap \{ u > 0 \}\right)  \mathrm{d}s
			\notag\\\geq
			& \frac{1}{2} 	\mathcal{L}^N\left( O_\sigma \cap B_r(x_0) \cap \{ u > 0 \}\right)
			\notag\\ \geq
			& \frac{1}{2} k \sigma r^{N-1}.\notag
		\end{align}
		Letting $k \to \infty$, we get a contradiction. Therefore, \eqref{8.1} holds.
		
		Now, according to Theorem~\ref{duo}, for $x_0 \in B_{1-\sigma} \cap \Upsilon^+$, there exist $y_0 \in \{ u > 0 \}$ and $c(N, p) > 0$ such that
		\begin{align}
			B_{c\sigma}(y_0) \subset \left( B_\sigma(x_0) \cap O_\sigma \cap \{ u > 0 \}\right) .\notag
		\end{align}
		By virtue of the Besicovitch covering theorem, let $\{ B_\sigma(x^i) \}_{i \in I}$ be finite coverings of $B_r(x_0) \cap \Upsilon^+$ with $x^i \in \Upsilon^+$ and at most $n(N)$ overlapping at each point. Then we obtain
		\begin{align}
			\sum_{i \in I} C(N) (c\sigma)^N
			\leq \sum_{i \in I} \mathcal{L}^N\left( O_\sigma \cap B_\sigma(x^i) \cap \{ u > 0 \}\right)
			\leq n \, \mathcal{L}^N\left( O_\sigma \cap B_r(x_0) \cap \{ u > 0 \}\right)
			\leq C n \sigma r^{N-1},\notag
		\end{align}
		where $C$ is a positive  constant  depending only on $N$, $p$, $\lambda$, $m_1$, $m_2$, $a_0$,   $a_1$,  $c_1$,  and $\Omega$.
		Therefore, it follows that
		\begin{align}
			\mathcal{H}^{N-1}\left( B_r(x_0) \cap \Upsilon^+\right)
			\leq \liminf_{\sigma \to 0} C(N) \sigma^{N-1}
			\leq C r^{N-1}.\notag
		\end{align}
			\end{proof}

\end{document}